\documentclass[11pt]{article}
\usepackage[width=15.5cm,height=21.5cm]{geometry}
\usepackage[colorlinks,allcolors=blue]{hyperref}

\usepackage{cite}

\newcommand{\doi}[1]{\href{https://doi.org/#1}{doi:#1}}
\newcommand{\myurl}[1]{\href{#1}{#1}}

\newcommand{\journaltitle}[1]{#1}
\newcommand{\booktitle}[1]{#1}

\newcommand{\yearvolumepages}[3]{\textbf{#2}, #3 (#1)}
\newcommand{\yearvolumeissuepages}[4]{\textbf{#2}:#3, #4 (#1)}

\usepackage{amsmath,amssymb}
\usepackage{mathtools}
\newcommand{\eqdef}{\coloneqq}
\usepackage{extarrows}

\newcommand{\bC}{\mathbb{C}}
\newcommand{\bD}{\mathbb{D}}
\newcommand{\bZ}{\mathbb{Z}}
\newcommand{\bN}{\mathbb{N}}
\newcommand{\bNz}{\mathbb{N}_0}
\newcommand{\bR}{\mathbb{R}}
\newcommand{\bS}{\mathbb{S}}
\newcommand{\bT}{\mathbb{T}}

\newcommand{\cA}{\mathcal{A}}
\newcommand{\cB}{\mathcal{B}}

\newcommand{\cG}{\mathcal{G}}
\newcommand{\cL}{\mathcal{L}}

\newcommand{\cR}{\mathcal{R}}
\newcommand{\cS}{\mathcal{S}}

\newcommand{\cU}{\mathcal{U}}

\newcommand{\cX}{\mathcal{X}}
\newcommand{\cY}{\mathcal{Y}}
\newcommand{\cTR}{\mathcal{TR}}
\newcommand{\cGTR}{\mathcal{GTR}}
\newcommand{\hatL}{\widehat{\cL}}

\newcommand{\Be}{\operatorname{B}}
\newcommand{\E}{\operatorname{E}}

\newcommand{\al}{\alpha}
\newcommand{\be}{\beta}
\newcommand{\Ga}{\Gamma}
\newcommand{\ga}{\gamma}
\newcommand{\de}{\delta}
\newcommand{\eps}{\varepsilon}

\newcommand{\ph}{\varphi}

\newcommand{\Om}{\Omega}
\newcommand{\om}{\omega}

\newcommand{\hatOm}{\widehat{\Om}}

\newcommand{\dif}{\mathrm{d}}

\newcommand{\Mat}{\mathcal{M}}

\newcommand{\lin}{\operatorname{span}}

\newcommand{\conjz}{\overline{z}}

\newcommand{\PureStates}{\operatorname{PS}}

\newcommand{\basic}[3]{b^{(#1)}_{#2,#3}}

\newcommand{\jac}{\mathcal{J}}
\newcommand{\jaccoef}[3]{c^{(#1,#2)}_{#3}}

\newcommand{\BL}{L^\infty_{\lim}([0,1))}
\newcommand{\BLZ}{L^\infty_{0}([0,1))}

\newcommand{\FreqSubspace}{\mathcal{W}}
\newcommand{\purestate}{\sigma}

\usepackage{colortbl}
\usepackage{xcolor}
\newcommand{\matr}[2]{\left[\begin{array}{#1}#2\end{array}\right]}

\newcommand{\bigleftparenthesis}{\left(\vphantom{%
\begin{bmatrix}
0 & 0 & 0 \\
0 & 0 & 0 \\
0 & 0 & 0
\end{bmatrix}
}\right.}

\newcommand{\bigrightparenthesis}{\left.\vphantom{%
\begin{bmatrix}
0 & 0 & 0 \\
0 & 0 & 0 \\
0 & 0 & 0
\end{bmatrix}
}\right)}

\colorlet{lightblue}{blue!20}

\usepackage{amsthm}
\newtheorem{thm}{Theorem}[section]
\newtheorem{prop}[thm]{Proposition}

\newtheorem{lem}[thm]{Lemma}

\theoremstyle{definition}
\newtheorem{example}[thm]{Example}
\newtheorem{defn}[thm]{Definition}
\newtheorem{remark}[thm]{Remark}

\author{Roberto Mois\'{e}s Barrera-Castel\'{a}n,
Egor A. Maximenko,
Gerardo Ramos-Vazquez}
\title{C*-algebras generated by
radial Toeplitz operators\\
on polyanalytic weighted Bergman spaces}

\begin{document}%
\maketitle

\begin{abstract}
In a previous paper
(Radial operators on polyanalytic weighted Bergman spaces,
Bol. Soc. Mat. Mex. 27, 43),
using disk polynomials as an orthonormal basis in the $n$-analytic weighted Bergman space, we showed that for every bounded radial generating symbol $a$,
the associated Toeplitz operator, acting in this space, 
can be identified with a matrix sequence $\gamma(a)$,
where the entries of the matrices are certain integrals involving $a$ and Jacobi polynomials.
In this paper, we suppose that the generating symbols $a$
have finite limits on the boundary and prove that the C*-algebra generated by the corresponding matrix sequences $\gamma(a)$ is the C*-algebra of all matrix sequences having scalar limits at infinity.
We use Kaplansky's noncommutative analog of the Stone--Weierstrass theorem
and some ideas from several papers by
Loaiza, Lozano,
Ram\'{i}rez-Ortega,
Ram\'{i}rez-Mora,
and S\'{a}nchez-Nungaray.
We also prove that for $n\ge 2$, the closure of the set of matrix sequences $\gamma(a)$ is not equal to the generated C*-algebra.

\medskip\noindent
\textbf{Keywords:}
Toeplitz operator,
radial operator,
polyanalytic function,
matrix sequence,
pure state,
Jacobi polynomial,
antitriangular matrix.

\medskip\noindent
\textbf{MSC (2020)}:
47L80, 47B35, 22D25, 30H20, 30G20, 33C45.

\end{abstract}

\bigskip
\section*{Financial support}

The first author has been supported by PhD scholarship
(CONACYT, Mexico).

\medskip\noindent
The second author has been supported by
CONACYT (Mexico) project ``Ciencia de Frontera'' FORDECYT-PRONACES/61517/2020
and by IPN-SIP projects
(Instituto Polit\'{e}cnico Nacional, Mexico).

\medskip\noindent
The third author has been supported by postdoctoral grant
(CONACYT, Mexico).

\clearpage

\tableofcontents

\bigskip\bigskip\bigskip
\section{Introduction and main results}

There are various deep results about Toeplitz operators with bounded symbols acting on (analytic) Bergman and Fock spaces,
including descriptions of the C*-algebras generated by Toeplitz operators with bounded symbols
\cite{BauerFulsche2020,Hagger2021,Xia2015}
and descriptions of Toeplitz operators invariant under some group actions
\cite{Vasilevski2008book,
GrudskyQuirogaVasilevski2006,
DawsonOlafssonQuiroga2015,
HerreraMaximenkoRamos2022,
QuirogaSanchez2021}.
In particular, in a series of
papers~\cite{Suarez2008,
GrudskyMaximenkoVasilevski2013,
BauerHerreraVasilevski2014,
HerreraVasilevskiMaximenko2015}
it was shown that the norm closure of the set of all radial Toeplitz operators with bounded measurable radial symbols is isomorphic and isometric to the C*-algebra of bounded sequences that slowly oscillate in the sense of Schmidt
(in other words, the sequences that are uniformly continuous with respect to the $\log$-distance).

In polyanalytic spaces,
the situation is usually more complicated.
The polyanalytic Bergman or Fock spaces and their applications have been thoroughly studied by many authors; see references in
\cite{AbreuFeichtinger2014,
LealMaximenkoRamos2021,
Youssfi2021}.

Recently, various authors investigated the behavior of vertical and angular Toeplitz operators acting on the polyanalytic Bergman or Fock spaces,
supposing that 
the generating symbols are bounded and have limits at the boundary of the domain
\cite{RamirezSanchez2015,
LoaizaRamirez2017,
SanchezGonzalezLopezArroyo2018,
RamirezRamirezSanchez2019,ArroyoSanchez2021}.
In particular, Ram\'{i}rez-Ortega and S\'{a}nchez-Nungaray proved~\cite{RamirezSanchez2015}
that the C*-algebra generated by such vertical Toeplitz operators in the $n$-analytic Bergman space over the upper half-plane is isomerically isomorphic to the C*-algebra of continuous matrix functions on $(0,+\infty)$ having scalar limits at $0$ and $+\infty$.
There are also similar results for C*-algebras of vertical and angular Toeplitz operators acting in harmonic and polyharmonic spaces~\cite{LoaizaLozano2013,LoaizaLozano2015,LoaizaMoralesRamirez2021}.
The authors of these papers used the technique of pure state separation and Kaplansky's noncommutative analog of Stone--Weierstrass theorem~\cite{Kaplansky1951}.
We also mention other papers about Toeplitz operators acting in polyanalytic Bergman spaces~\cite{CuckovicLe2012,HutnikHutnikova2015,RozenblumVasilevski2018,Vasilevski1999,RamirezRamirezMorales2020,
Hagger2022}.

In~\cite{BMR2021},
we proved that the radial operators acting in the $n$-analytic Bergman space $\cA^2(\bD,\mu_\al)$
can be identified with bounded matrix families consisting of square complex matrices of orders $1,\ldots,n,n,n,\ldots$:
\[
A=(A_\xi)_{\xi=-n+1}^\infty,\qquad
A_\xi\in\Mat_{\min\{n+\xi,n\}},\qquad
\sup_{-n+1<\xi<+\infty}\|A_\xi\|<+\infty.
\]
In other words, the von Neumann algebra of all bounded radial operators acting in $\cA^2(\bD,\mu_\al)$
is spatially equivalent to the following direct sum of matrix algebras:
\begin{equation}
\label{eq:cS_def_intro}
\cS_n
\eqdef
\bigoplus_{\xi=-n+1}^\infty
\Mat_{\min\{n+\xi,n\}}
=\Mat_1\oplus\Mat_2\oplus\cdots\oplus\Mat_{n-1}\oplus
\Mat_n\oplus\Mat_n\oplus\Mat_n\oplus\cdots.
\end{equation}
We use the term ``frequency'' for the variable $\xi$ because~\eqref{eq:cS_def_intro} is the ``Fourier decomposition'' of the von Neumann algebra of radial operators.
More formally, we deal with a unitary representation of the group $\bT=\{z\in\bC\colon\ |z|=1\}$, and $\xi$ runs through a subset of the dual group $\bZ$.

In particular, if $a\in L^\infty([0,1))$ and $\widetilde{a}(z)=a(|z|)$,
then the radial Toeplitz operator $T_{n,\al,\widetilde{a}}$
acting in $\cA_n^2(\bD,\mu_\al)$
can be identified with the matrix sequence
$\ga_{n,\al}(a)\in\cS_n$,
where the components of the matrices $\ga_{n,\al}(a)_\xi$
are explicitly given as certain integrals involving $a$ and Jacobi polynomials, see Section~\ref{sec:radial_Toeplitz_operators_as_matrix_sequences}.
The situation is similar for radial operators in the Bargmann--Segal--Fock space~\cite{MaximenkoTelleriaRomero2020}, but the Laguerre polynomials are used there instead of Jacobi polynomials.

In this paper, we continue our investigation~\cite{BMR2021}.
Let $\BL$ be the vector space of essentially bounded functions having a limit at $1$:
\begin{equation}
\label{eq:BL_def}
\BL\eqdef
\Bigl\{a\in L^\infty([0,1))\colon\quad
\exists \om\in\bC\quad
\lim_{r\to 1}a(r)=\om
\Bigr\}.
\end{equation}
We denote by $\cGTR_{n,\al}$ the set (vector space) of all radial Toeplitz operators
$T_{n,\al,\widetilde{a}}$
associated to such generating symbols and by $\cTR_{n,\al}$ the C*-algebra generated by $\cGTR_{n,\al}$.

On the other hand, we denote by $\cG_{n,\al}$ the set of the
matrix families $\ga_{n,\al}(a)$ corresponding to such radial Toeplitz operators:
\begin{equation}
\label{eq:cG_def}
\cG_{n,\al}
\eqdef
\Bigl\{\ga_{n,\al}(a)\colon\ a\in\BL\Bigr\}.
\end{equation}
Let $\cX_{n,\al}$ be the C*-algebra generated by $\cG_{n,\al}$.
The objective of this paper is to describe $\cX_{n,\al}$.

Let $\cL_n$ be the subalgebra of $\cS_n$ consisting of all matrix sequences
having scalar limits at infinity:
\begin{equation}
\label{eq:cL_def}
\cL_n \eqdef
\Bigl\{
A\in\cS_n\colon\quad
\exists\om\in\bC\quad
\lim_{\xi\to\infty}A_\xi=\om I_n
\Bigr\}.
\end{equation}
Here are the main results of the paper.

\begin{thm}
\label{thm:main}
$\cX_{n,\al}=\cL_n$.
\end{thm}

\begin{thm}
\label{thm:closure_is_not_Cstar_algebra}
For $n\ge 2$,
the closure of $\cG_{n,\al}$ in $\cL_n$ does not coincide with $\cL_n$.
\end{thm}

For $n=1$ (analytic case), $\cL_n$ is just the space $c(\bNz)$ of converging sequences,
and we prove that $\cG_{1,\al}$ is dense in $\cL_1=c(\bNz)$
(Proposition~\ref{prop:G1_is_dense_in_c}).
Here, the situation is similar to~\cite{BauerFulsche2020,Hagger2021,Xia2015}
and~\cite{Suarez2008,
GrudskyMaximenkoVasilevski2013,
BauerHerreraVasilevski2014,
HerreraVasilevskiMaximenko2015}, where the generated C*-algebra coincides with the closure of the generating subspace.

For $n\ge 2$, Theorems~\ref{thm:main} and~\ref{thm:closure_is_not_Cstar_algebra} mean that $\cTR_{n,\al}$
can be identified with $\cL_n$,
and the closure of $\cGTR_{n,\al}$
in the uniform operator topology
is a proper subspace of $\cTR_{n,\al}$.
So, the $n$-analytic case with $n\ge 2$ is very different from the analytic case ($n=1$).

It is easy to see that $\cX_{n,\al}\subseteq\cL_n$. The difficult part is to show that $\cG_{n,\al}$ is sufficiently rich to generate $\cL_n$.
Following Loaiza and Lozano~\cite{LoaizaLozano2013} and other papers \cite{RamirezSanchez2015,
LoaizaRamirez2017,
SanchezGonzalezLopezArroyo2018,
RamirezRamirezSanchez2019,
RamirezRamirezMorales2020},
we use the concept of pure states
and apply Kaplansky's noncommutative analog of Stone--Weierstrass theorem to prove that
$\cX_{n,\al}=\cL_n$.

One of the main ideas of the paper is to consider the matrix sequences $\ga_{n,\al}(a)$ taking as generating symbols $a$ some Jacobi polynomials, with an appropriate change of variables.  
Due to orthogonal properties of Jacobi polynomials,
the corresponding matrices $\ga_{n,\al}(a)_\xi$ have a certain ``antitriangular structure'' and provide a convenient set of generators.
We think that similar ideas can be useful for other reproducing kernel Hilbert spaces and other group actions.

Another crucial idea of this paper is a special scheme to generate $\Mat_n$ from a finite matrix set.
This idea was briefly explained by Ram\'{i}rez~Ortega, Ram\'{i}rez~Mora, and S\'{a}nchez-Nungaray in~\cite{RamirezRamirezSanchez2019},
but we develop it in a more detailed form in Section~\ref{sec:one_generating_set_of_matrices}.

There are some differences between this paper and~\cite{RamirezSanchez2015,
LoaizaRamirez2017,
SanchezGonzalezLopezArroyo2018,
RamirezRamirezSanchez2019,
RamirezRamirezMorales2020}.
\begin{itemize}
\item In this paper, the components of the matrices $\ga_{n,\al}(a)_\xi$ and $\ga_{n,\al}(a)_\eta$, with $\xi\ne\eta$,
have a very similar form:
$\xi$ and $\eta$ only affect some powers and coefficients.
Therefore, the corresponding pure states are difficult to separate (see Proposition~\ref{prop:pure_states_can_coincide_on_gammas}).

\item If $\xi<0$, then the order of the matrix $\ga_{n,\al}(a)_\xi$ is different from $n$
and depends on $\xi$.
Thus, the treatment of the matrices $\ga_{n,\al}(a)_\xi$ corresponding to negative frequencies $\xi$
requires an additional work.
\end{itemize}

By analogy to~\cite{RamirezRamirezSanchez2019,RamirezRamirezMorales2020},
it is natural to conjecture that $\cX_{n,\al}$ can be generated by a finite set of matrix sequences $\ga_{n,\al}(a)$ with $a$ in $\BL$,
but we have not proved (or disproved) this conjecture.

The paper has the following structure.
In Sections~\ref{sec:Jacobi}
and~\ref{sec:radial_Toeplitz_operators_as_matrix_sequences},
we recall the main properties of Jacobi polynomials on $(0,1)$ and the correspondence between the radial Toeplitz operators acting in $\cA_n^2(\bD,\mu_\al)$ and matrix sequences.
In Section~\ref{sec:limits_of_matrix_sequences}, we show that $\cX_{n,\al}\subseteq\cL_n$.
Section~\ref{sec:algebra_of_matrix_sequences_with_scalar_limits} contains a description of the pure states of $\cL_n$.
In Section~\ref{sec:one_generating_set_of_matrices}, we explain how to construct the matrix algebra $\Mat_n$ from a special generating set.
In Sections~\ref{sec:algebra_gammas_fixed_frequency}, \ref{eq:separate_limit_value},
\ref{sec:separate_pure_states_associated_to_different_frequencies}, and~\ref{sec:separate_pure_states_associated_to_different_frequencies_when_the_lower_frequency_is_negative}, we prove that $\cX_{n,\al}$ separates the pure states of $\cL_n$.
In Section~\ref{sec:proofs_of_main_results}, we finish the proofs of the main results.

\section{Shifted Jacobi polynomials on the unit interval}
\label{sec:Jacobi}

Let $\al,\be>-1$.
For every $m$ in $\bNz\eqdef\{0,1,2,\ldots\}$,
we denote by $Q_m^{(\al,\be)}$
the ``shifted Jacobi polynomial''
obtained from the classical Jacobi polynomial
$P_m^{(\al,\be)}$
by composing it with the change of variables $t\mapsto 2t-1$:
\[
Q_m^{(\al,\be)}(t)\eqdef P_m^{(\al,\be)}(2t-1).
\]
Properties of the polynomials $Q_m^{(\al,\be)}$ can be easily deduced from well-known properties of $P_m^{(\al,\be)}$, see~\cite[Chapter~4]{Szego1975}
and~\cite[Section~2]{BMR2021}.
Here is the Rodrigues-type formula and the explicit expression for
$Q_m^{(\al,\be)}$:
\begin{align}
\label{eq:shifted_Jacobi_Rodrigues}
Q_m^{(\al,\be)}(t)
&=\frac{(-1)^m}{m!}\,(1-t)^{-\al} t^{-\be}\,
\frac{\dif^m}{\dif{}t^m}
\Bigl((1-t)^{m+\al} t^{m+\be}\Bigr),
\\
\label{eq:shifted_Jacobi_explicit}
Q_m^{(\al,\be)}(t)
&=\sum_{k=0}^m
\binom{\al+\be+m+k}{k}
\binom{\be+m}{m-k}
(-1)^{m-k} t^k.
\end{align}
The sequence $(Q_m^{(\al,\be)})_{m=0}^\infty$
is orthogonal on $(0,1)$ with respect to the
``Jacobi weight'' $t\mapsto (1-t)^\al t^\be$.
More precisely, for every $p,q$ in $\bNz$,
\begin{equation}
\label{eq:Q_inner_prod}
\int_0^1
Q_p^{(\al,\be)}(t)
Q_q^{(\al,\be)}(t)
\,
(1-t)^\al t^\be\,\dif{}t
=\de_{p,q}\,
\frac{\Ga(p+\al+1)\Ga(p+\be+1)}%
{(2p+\al+\be+1)\Ga(p+\al+\be+1)\,p!},
\end{equation}
where $\de_{p,q}$ is Kronecker's delta.
In this paper, we will also use the 
orthogonal properties of the sequence $(Q_m^{(\al,\be)})_{m=0}^\infty$ in the following form
(these properties can be proved applying~\eqref{eq:shifted_Jacobi_Rodrigues} and integration by parts).

\begin{prop}
\label{prop:int_Q_by_polynomial}
Let $\al>-1$ and $\be>-1$.
Then for every $m$ in $\bNz$
and every polynomial $f$ of degree $<m$,
\[
\int_0^1 Q_m^{(\al,\be)}(t) f(t) \,(1-t)^\al t^\be\,\dif{}t
=0,
\]
but for every polynomial $f$ of degree $m$,
\[
\int_0^1 Q_m^{(\al,\be)}(t) f(t) \,(1-t)^\al t^\be\,\dif{}t
\ne 0.
\]
\end{prop}

In particular, for $f(t)=t^m$,
the integral from Proposition~\ref{prop:int_Q_by_polynomial} can be computed explicitly:
\[
\int_0^1 Q_m^{(\al,\xi)}(t)\,
t^{m} (1-t)^\al t^{\xi}\,\dif{}t
=\Be(\xi+m+1,\al+m+1).
\]
Inspired by~\eqref{eq:Q_inner_prod}
we define the ``Jacobi function''
$\jac_m^{(\al,\be)}$ on $(0,1)$ as
\begin{equation}
\label{eq:jac}
\jac_m^{(\al,\be)}(t)
\eqdef \jaccoef{\al}{\be}{m}
(1-t)^{\al/2}t^{\be/2}Q_m^{(\al,\be)}(t),
\end{equation}
where
\begin{equation}
\label{eq:jaccoef}
\jaccoef{\al}{\be}{m}
\eqdef
\sqrt{\frac{(2m+\al+\be+1)\,\Ga(m+\al+\be+1)\,m!}
{\Ga(m+\al+1)\Ga(m+\be+1)}}.
\end{equation}
For $\be$ in $\bNz$,
the normalizing coefficient
$\jaccoef{\al}{\be}{m}$
can be expressed through some binomial coefficients:
\begin{equation}
\label{eq:jaccoef_via_binomial}
\jaccoef{\al}{\be}{m}
=\sqrt{\frac{(2m+\al+\be+1)\,
\binom{\al+\be+1}{\be}}%
{\binom{m+\be}{\be}}}.
\end{equation}
The function sequence $(\jac_m^{(\al,\be)})_{m=0}^\infty$ is orthonormal on $(0,1)$ without weight:
\[
\int_0^1
\jac_p^{(\al,\be)}(t)
\jac_q^{(\al,\be)}(t)
\,\dif{}t
=\de_{p,q}.
\]

\section{Representation of radial Toeplitz operators as matrix sequences}
\label{sec:radial_Toeplitz_operators_as_matrix_sequences}

In this section,
we recall the main results from~\cite{BMR2021} and slightly change some notation.

Let $n\in\bN\eqdef\{1,2,\ldots\}$ and $\al>-1$.
We denote by $\cA_n^2(\bD,\mu_\al)$
the space of $n$-analytic functions,
square integrable with respect to the normalized weighted Lebesgue plane measure
\[
\dif\mu_\al(z)\eqdef
\frac{\al+1}{\pi}\,(1-|z|^2)^\al\,\dif\mu(z).
\]
For every $\tau$ in the unit circle
$\bT\eqdef\{z\in\bC\colon\ |z|=1\}$,
let $\rho_{n,\al}(\tau)$ be the rotation
operator acting in $\cA_n^2(\bD,\mu_\al)$ by the rule
\[
(\rho_{n,\al}(\tau)f)(z)\eqdef f(\tau^{-1}z).
\]
The family $\rho_{n,\al}$
is a unitary representation of the group $\bT$
in the space $\cA_n^2(\bD,\mu_\al)$.
We denote by $\cR_{n,\al}$ its centralizer (commutant),
i.e., the von Neumann algebra that consists
of all bounded linear operators
acting in $\cA_n^2(\bD,\mu_\al)$
that commute with $\rho_{n,\al}$
for every $\tau$ in $\bT$.
In other words, the elements of $\cR_{n,\al}$
are the operators
intertwining the representation $\rho_{n,\al}$.
The elements of $\cR_{n,\al}$ are called
\emph{radial} operators in $\cA_n^2(\bD,\mu_\al)$.

In~\cite{BMR2021}, we explained various equivalent definitions of the normalized disk polynomials
$\basic{\al}{p}{q}$.
In particular, we expressed them via the shifted Jacobi polynomials:
\begin{align}
\label{eq:b_via_Q}
\basic{\al}{p}{q}(r\tau)
&=\frac{\jaccoef{\al}{|p-q|}{\min\{p,q\}}}{\sqrt{\al+1}}
r^{|p-q|}\tau^{p-q}
Q_{\min\{p,q\}}^{(\al,|p-q|)}(r^2)
\qquad(0\le r<1,\ \tau\in\bT),
\\
\label{eq:b_via_jac}
\basic{\al}{p}{q}(r\tau)
&=\frac{\tau^{p-q}(1-r^2)^{-\al/2}}{\sqrt{\al+1}}
\jac_{\min\{p,q\}}^{(\al,|p-q|)}(r^2)
\qquad(0\le r<1,\ \tau\in\bT).
\end{align}
For every $\xi$ in $\bZ$ and every $s$ in $\bN$, we denote by $\FreqSubspace^{(\al)}_{\xi,s}$ the subspace of $L^2(\bD,\mu_\al)$
generated by the monomial functions
$z\mapsto z^p \conjz^q$
with $p-q=\xi$ and
$\min\{p,q\}\leq s$.
Equivalently,
\begin{equation}\label{eq:truncated_FreqSbpce_gen_by_b}
\FreqSubspace^{(\al)}_{\xi,s}
=\lin
\bigl\{\basic{\al}{p}{q}\colon\ p-q=\xi,\ 
\min\{p,q\}\leq s\bigr\}.
\end{equation}
The vector space $\FreqSubspace^{(\al)}_{\xi,s}$ does not depend on $\al$, but we endow it with the inner product from $L^2(\bD,\mu_\al)$.
Obviously, $\dim(\FreqSubspace^{(\al)}_{\xi,s})=s$.
As it was exposed in \cite{BMR2021},
\begin{itemize}
\item the family $(\basic{\al}{p}{q})_{p,q\in\bNz}$ is
an orthonormal basis of $L^{2}(\bD, \mu_{\al})$,
\item the family $(b_{p,q})_{p\in\bNz,0\le q<n}$
is an orthonormal basis of $\cA_n^2(\bD,\mu_\al)$.
\end{itemize}
When $p,q\in\bNz$ and $q<n$, the difference $p-q$ belongs to the set
\begin{equation}
\label{eq:Om_def}
\Om_n\eqdef\{\xi\in\bZ\colon\ \xi\ge-n+1\}
=\{-n+1,\ldots,0,1,\ldots\}.
\end{equation}
Reasoning in terms of orthonormal bases we decompose $\cA_n^{2}(\bD,\mu_\al)$ into the following direct sum of the ``truncated frequency subspaces'':
\[
\cA_n^{2}(\bD,\mu_\al)
=\bigoplus_{\xi\in\Om_n} \FreqSubspace^{(\al)}_{\xi,\min\{n+\xi,n\}}.
\]

\begin{remark}
In what follows,
we introduce a notation slightly different from~\cite{BMR2021}.
In~\cite{BMR2021}, for $\xi<0$, 
we used indices from $-\xi$ to $n-1$ to numerate the elements of the basis of
$\FreqSubspace^{(\al)}_{\xi,n+\xi}$
and the components of the matrices $\Phi_n(S)_\xi$ and
$\ga_{n,\al}(a)_\xi$ which appear below.
In this paper, for $\xi<0$ we will use indices from $0$ to $n+\xi-1$.
\end{remark}

For the orthonormal basis of $\FreqSubspace^{(\al)}_{\xi,\min\{n,n+\xi\}}$
given in~\eqref{eq:truncated_FreqSbpce_gen_by_b},
we employ the following enumeration:
\begin{equation}\label{eq:truncated_FreqSbpce_basis}
\basic{\al}{\max\{j+\xi,j\}}{\max\{j-\xi,j\}},
\qquad
0\le j\le \min\{n+\xi,n\}-1.
\end{equation}
Then, we obtain an isometric isomorphism
\[
U_{n,\al}\colon
\cA_n^2(\bD,\mu_\al)
\to 
\bigoplus_{\xi\in\Om_n}
\bC^{\min\{n+\xi,n\}},
\]
acting by the following rule:
\[
(U_{n,\al} f)_{\xi,j}
\eqdef \langle f, \basic{\al}{\max\{j+\xi,j\}}{\max\{j-\xi,j\}}\rangle.
\]
Recall that $\cS_n$ is defined in~\eqref{eq:cS_def_intro}.
Equivalently,
\[
\cS_n \eqdef\bigoplus_{\xi\in\Om_n}
\Mat_{\min\{n,n+\xi\}}
= \left(\bigoplus_{\xi=-n+1}^{-1}
\Mat_{n+\xi}\right)
\oplus
\left(\bigoplus_{\xi=0}^{\infty}
\Mat_n\right).
\]
For example,
\[
\cS_3
=\underbrace{\Mat_1}_{\xi=-2}
\oplus\underbrace{\Mat_2}_{\xi=-1}
\oplus\underbrace{\Mat_3}_{\xi=0}
\oplus\underbrace{\Mat_3}_{\xi=1}
\oplus\underbrace{\Mat_3}_{\xi=2}
\oplus\dots.
\]
According to the definition of the direct sum,
$\cS_n$ consists of all matrix sequences
of the form
$A=(A_\xi)_{\xi\in\Om_n}$,
where $A_\xi\in\Mat_{n+\xi}$ if $\xi<0$, $A_\xi\in\Mat_n$ if $\xi\ge0$, and
\[
\sup_{\xi\in\Om_n}\|A_\xi\|<+\infty.
\]
Being a direct sum of W*-algebras,
$\cS_n$ is a W*-algebra.
Moreover, $\cS_n$ can be seen as a von Neumann algebra of operators acting in the Hilbert space
$\bigoplus_{\xi\in\Om_n}
\bC^{\min\{n+\xi,n\}}$.

Define $\Phi_{n,\al}\colon\cR_{n,\al}\to\cS_n$ by the following rules.
For $\xi$ in $\bNz$,
\begin{equation}
\label{eq:Phi_positive}
\Phi_{n,\al}(S)_\xi
=\left[\left\langle S\basic{\al}{\xi+k}{k},\basic{\al}{\xi+j}{j}\right\rangle\right]_{j,k=0}^{n-1}.
\end{equation}
For $\xi$ in $\{-n+1,\ldots,-1\}$,
\begin{equation}
\label{eq:Phi_negative}
\Phi_{n,\al}(S)_\xi
=\left[\left\langle S\basic{\al}{k}{k-\xi},\basic{\al}{j}{j-\xi}\right\rangle\right]_{j,k=0}^{n+\xi-1}.
\end{equation}
To join these two cases 
(\eqref{eq:Phi_positive} and~\eqref{eq:Phi_negative}),
we enumerate the elements of the basis as in~\eqref{eq:truncated_FreqSbpce_basis}.
Then, the whole sequence $\Phi_n(S)$ can be written in the following form:
\begin{equation}
\label{eq:Phi}
\Phi_{n,\al}(S)
\eqdef
\left(\left[\left\langle S\basic{\al}{\max\{k+\xi,k\}}{\max\{k-\xi,k\}},\basic{\al}{\max\{j+\xi,j\}}{\max\{j-\xi,j\}}\right\rangle\right]_{j,k=0}^{\min\{n+\xi,n\}-1}
\right)_{\xi\in\Om_n}.
\end{equation}
Then $\Phi_{n,\al}$ is an isomorphism of W*-algebras.
Moreover, $\Phi_{n,\al}$ is induced by $U_{n,\al}$ in the following sense:
if $S\in\cR_{n,\al}$ and $f\in\cA_n^2(\bD,\mu_\al)$, then
\[
U_{n,\al}(S f)
=\Phi_{n,\al}(S) U_{n,\al}(f).
\]
Hence, the von Neumann algebras $\cR_{n,\al}$ and $\cS_n$ are spatially isomorphic:
\[
\cR_{n,\al}
\cong\bigoplus_{\xi=-n+1}^\infty \cB(\FreqSubspace^{(\al)}_{\xi,\min\{n+\xi,n\}})
\cong
\cS_n.
\]

\begin{example}
\label{example:Phi_matrices_for_n_equal_3}
For $n=3$ and $\xi\ge0$,
the list $(\basic{\al}{\xi}{0}, \basic{\al}{\xi+1}{1}, \basic{\al}{\xi+2}{2})$ is an ordered basis of
$\FreqSubspace^{(\al)}_{\xi,3}$.
According to~\eqref{eq:Phi_positive} or~\eqref{eq:Phi},
\[
\Phi_3(S)_\xi
=
\begin{bmatrix}
\langle S\basic{\al}{\xi}{0},
\basic{\al}{\xi}{0} \rangle
&
\langle S\basic{\al}{\xi+1}{1},
\basic{\al}{\xi}{0} \rangle
&
\langle S\basic{\al}{\xi+2}{2},
\basic{\al}{\xi}{0} \rangle
\\[1ex]
\langle S\basic{\al}{\xi}{0},
\basic{\al}{\xi+1}{1} \rangle
&
\langle S\basic{\al}{\xi+1}{1},
\basic{\al}{\xi+1}{1} \rangle
&
\langle S\basic{\al}{\xi+2}{2},
\basic{\al}{\xi+1}{1} \rangle
\\[1ex]
\langle S\basic{\al}{\xi}{0},
\basic{\al}{\xi+2}{2} \rangle
&
\langle S\basic{\al}{\xi+1}{1},
\basic{\al}{\xi+2}{2} \rangle
&
\langle S\basic{\al}{\xi+2}{2},
\basic{\al}{\xi+2}{2} \rangle
\end{bmatrix}.
\]
In particular, for $n=3$ and $\xi=1$,
the list $(\basic{\al}{1}{0}, \basic{\al}{2}{1}, \basic{\al}{3}{2})$ is an ordered basis of
$\FreqSubspace^{(\al)}_{1,3}$,
and
\[
\Phi_3(S)_1
=
\begin{bmatrix}
\langle S\basic{\al}{1}{0},
\basic{\al}{1}{0} \rangle
&
\langle S\basic{\al}{2}{1},
\basic{\al}{1}{0} \rangle
&
\langle S\basic{\al}{3}{2},
\basic{\al}{1}{0} \rangle
\\[1ex]
\langle S\basic{\al}{1}{0},
\basic{\al}{2}{1} \rangle
&
\langle S\basic{\al}{2}{1},
\basic{\al}{2}{1} \rangle
&
\langle S\basic{\al}{3}{2},
\basic{\al}{2}{1} \rangle
\\[1ex]
\langle S\basic{\al}{1}{0},
\basic{\al}{3}{2} \rangle
&
\langle S\basic{\al}{2}{1},
\basic{\al}{3}{2} \rangle
&
\langle S\basic{\al}{3}{2},
\basic{\al}{3}{2} \rangle
\end{bmatrix}.
\]
For $n=3$ and $\xi=-1$,
the list
$(\basic{\al}{0}{1}, \basic{\al}{1}{2})$
is an ordered basis of $\FreqSubspace^{(\al)}_{-1,2}$.
By~\eqref{eq:Phi_negative},
\[
\Phi_3(S)_{-1}
=
\begin{bmatrix}
\langle S\basic{\al}{0}{1},
\basic{\al}{0}{1} \rangle
&
\langle S\basic{\al}{1}{2},
\basic{\al}{0}{1} \rangle
\\[1ex]
\langle S\basic{\al}{0}{1},
\basic{\al}{1}{2} \rangle
&
\langle S\basic{\al}{1}{2},
\basic{\al}{1}{2} \rangle
\end{bmatrix}.
\]
\end{example}

\subsection*{Radial Toeplitz operators
in $\cA_n^2(\bD,\mu_\al)$}

Given $b$ in $L^\infty(\bD)$, let $T_{n,\al,b}\colon \cA_n^2(\bD,\mu_\al)\to \cA_n^2(\bD,\mu_\al)$ be the Toeplitz operator with generating symbol $b$.
It is easy to see that $T_{n,\al,b}$ is radial if and only if $b$ is a radial function (more precisely, $b$ is a class of equivalence containing a radial function).

Given $a$ in $L^\infty([0,1))$,
$\xi$ in $\Om_n$, and $j,k$ in $\bNz$,
we denote by $\be_{a,\al,\xi,j,k}$
the following number:
\begin{equation}\label{eq:gamma_as_integral_jac}
\be_{a,\al,\xi,j,k}
\eqdef
\int_0^1 a(\sqrt{t})\,
\jac_{j}^{(\al,|\xi|)}(t)\,
\jac_{k}^{(\al,|\xi|)}(t)
\,\dif{}t.
\end{equation}
Equivalently,
\begin{equation}
\label{eq:gamma_as_integral}
\be_{a,\al,\xi,j,k}
=
\jaccoef{\al}{|\xi|}{j}
\jaccoef{\al}{|\xi|}{k}
\int_0^1 a(\sqrt{t})\,
Q_{j}^{(\al,|\xi|)}(t)\,
Q_{k}^{(\al,|\xi|)}(t)
\,(1-t)^\al\,t^{|\xi|}\,\dif{}t.
\end{equation}
Given $a$ in $L^\infty([0,1))$,
we denote by $\ga_{n,\al}(a)$
the matrix sequence
$[\ga_{n,\al}(a)_\xi]_{\xi\in\Om_n}$,
where
$\ga_{n,\al}(a)_\xi\in\Mat_{\min\{n+\xi,n\}}$
and
\begin{equation}
\label{eq:gamma_def}
\ga_{n,\al}(a)_\xi
\eqdef
\bigl[\be_{a,\al,\xi,j,k}\bigr]_{j,k=0}^{\min\{n+\xi,n\}-1}.
\end{equation}
Let $a\in L^\infty([0,1))$
and $\widetilde{a}\in L^\infty(\bD)$ be defined by
$\widetilde{a}(z)\eqdef a(|z|)$.

By~\eqref{eq:b_via_Q} or~\eqref{eq:b_via_jac},
\begin{equation}
\label{eq:radial_Toeplitz_applied_to_the_basis}
\left\langle
T_{n,\al,\widetilde{a}}\,
\basic{\al}{\max\{k+\xi,k\}}{\max\{k-\xi,k\}},\basic{\al}{\max\{j+\eta,j\}}{\max\{j-\eta,j\}}\right\rangle
=
\de_{\xi,\eta}
\be_{a,\al,\xi,j,k}.
\end{equation}
Formula~\eqref{eq:radial_Toeplitz_applied_to_the_basis} implies that
$T_{n,\al,\widetilde{a}}\in\cR_{n,\al}$ and
\[
U_{n,\al}(T_{n,\al,\widetilde{a}} f)
=\ga_{n,\al}(a)
U_{n,\al}(f)
\qquad(f\in\cA_n^2(\bD,\mu_\al)),
\]
i.e.,
\[
\Phi_{n,\al}(T_{n,\al,\widetilde{a}})
=\ga_{n,\al}(a).
\]
Thereby, the study of radial Toeplitz operators $T_{n,\al,\widetilde{a}}$ acting in $\cA_n^2(\bD,\mu_\al)$ is reduced to the study of the matrix sequences $\ga_{n,\al}(a)$.

\begin{remark}
\label{rem:gammas_are_symmetric}
For all $\al,a,\xi$,
the matrices
$\ga_{n,\al}(a)$ are symmetric
because
$\be_{a,\al,\xi,j,k}
=\be_{a,\al,\xi,k,j}$:
\[
\ga_{n,\al}(a)^\top
=\ga_{n,\al}(a).
\]
\end{remark}

\begin{remark}
\label{rem:gammas_for_negative_frequencies}
For $\xi$ in $\{-n+1,\ldots,-1\}$,
according to~\eqref{eq:gamma_as_integral}
and~\eqref{eq:gamma_def},
\[
\ga_{n,\al}(a)_\xi\in\Mat_{n+\xi},\qquad
\ga_{n,\al}(a)_{|\xi|}\in\Mat_n,
\]
and $\ga_{n,\al}(a)_\xi$
is the upper-left submatrix of
$\ga_{n,\al}(a)_{|\xi|}$.
For example ($n=4$),
\[
\ga_{4,\al}(a)_2
=
\begin{bmatrix}
\be_{a,\al,2,0,0}
&
\be_{a,\al,2,0,1}
&
\be_{a,\al,2,0,2}
&
\be_{a,\al,2,0,3}
\\
\be_{a,\al,2,0,1}
&
\be_{a,\al,2,1,1}
&
\be_{a,\al,2,1,2}
&
\be_{a,\al,2,1,3}
\\
\be_{a,\al,2,0,2}
&
\be_{a,\al,2,1,2}
&
\be_{a,\al,2,2,2}
&
\be_{a,\al,2,2,3}
\\
\be_{a,\al,2,0,3}
&
\be_{a,\al,2,1,3}
&
\be_{a,\al,2,2,3}
&
\be_{a,\al,2,3,3}
\end{bmatrix},
\quad
\ga_{4,\al}(a)_{-2}
=
\begin{bmatrix}
\be_{a,\al,2,0,0}
&
\be_{a,\al,2,0,1}
\\
\be_{a,\al,2,0,1}
&
\be_{a,\al,2,1,1}
\end{bmatrix}.
\]
\end{remark}

Recall that $\BL$ is defined by~\eqref{eq:BL_def}.
In this paper, we consider the Toeplitz operators $T_{n,\al,\widetilde{a}}$,
where $a\in\BL$.
Let $\cGTR_{n,\al}$ be the set of all such operators:
\[
\cGTR_{n,\al}
\eqdef
\bigl\{T_{n,\al,\widetilde{a}}\colon\
a\in\BL\bigr\}.
\]
Let $\cTR_{n,\al}$ be the C*-algebra generated by $\cGTR_{n,\al}$.
Furthermore, we denote by $\cG_{n,\al}$ the set of the matrix sequences $\ga_{n,\al}(a)$, see~\eqref{eq:cG_def},
and by $\cX_{n,\al}$ the C*-algebra generated by $\cG_{n,\al}$.
The C*-algebra $\cTR_{n,\al}$
is spatially isomorphic to $\cX_{n,\al}$.
In the rest of the paper, we deal with $\cG_{n,\al}$ and $\cX_{n,\al}$ instead of working directly with the operators $T_{n,\al,\widetilde{a}}$.

\section{Limits of matrix sequences associated to radial Toeplitz operators}
\label{sec:limits_of_matrix_sequences}

The main result of this section is Proposition~\ref{prop:a_has_limit_implies_gamma_has_limit}.
We need a couple of elementary inequalities
for the binomial coefficients and the Gamma function:
\begin{equation}\label{eq:binom_upper_bound}
\binom{z+k}{k}
=\frac{(z+k)(z+k-1)\cdots (z+1)}{k!}
\le \frac{(z+k)^k}{k!}\qquad(z>0,\ k\in\bN_0),
\end{equation}
\begin{equation}\label{eq:Wendel_rough}
\frac{\Ga(z+a)}{\Ga(z)}\le (z+a)^a\qquad(z>0,\ a>0).
\end{equation}
Wendel~\cite{Wendel1948} proved that
$\Ga(z+a)\le z^a \Ga(z)$ for $z>0$ and $0<a<1$.
The rough upper bound~\eqref{eq:Wendel_rough}
follows from Wendel's inequality
and the recurrence relation $\Ga(z+1)=z\Ga(z)$.
See various proofs and reviews of inequalities for Gamma function ratios in \cite{Qi2010,Jameson2013}.

\begin{lem}\label{lem:jac_tends_to_zero}
Let $m\in\bNz$ and $0<x<1$. Then
\[
\lim_{\be\to+\infty}
\sup_{0\le t\le x}|\jac_m^{(\al,\be)}(t)|=0.
\]
\end{lem}

\begin{proof}
We recall that $\jac_m^{(\al,\be)}$
is defined by~\eqref{eq:jac}.
Consider the case $\al>0$ only,
since the case $-1<\al\le0$ is simpler.
First, we use~\eqref{eq:Wendel_rough} and
estimate from above the coefficient
$\jaccoef{\al}{\be}{n}$ defined by~\eqref{eq:jaccoef}:
\begin{align*}
\left(\jaccoef{\al}{\be}{m}\right)^2
&=(2m+\al+\be+1)
\cdot\frac{\Ga(m+\al+\be+1)}{\Ga(m+\be+1)}
\cdot\frac{m!}{\Ga(m+\al+1)}
\\
&\le (2m+\al+\be+1) (m+\al+\be+1)^\al
\le (2m+\al+\be+1)^{\al+1}.
\end{align*}
Next, we estimate the values of
$Q_m^{(\al,\be)}(t)$ with the help of~\eqref{eq:shifted_Jacobi_explicit} and~\eqref{eq:binom_upper_bound}:
\begin{align*}
|Q_m^{(\al,\be)}(t)|
&\le 
\sum_{k=0}^m
\binom{\al+\be+m+k}{k}\binom{\be+m}{m-k} t^k
\le 
\sum_{k=0}^m
\frac{(\al+\be+m+k)^k}{k!}
\frac{(\be+m)^{m-k}}{(m-k)!}
\\
&\le (m+1) (2m+\al+\be)^m
< (2m+\al+\be+1)^{m+1}.
\end{align*}
Therefore, for $0\le t\le x$,
\begin{align*}
|\jac_m^{(\al,\be)}(t)|
&=\jaccoef{\al}{\be}{m}
(1-t)^{\al/2}t^{\be/2} |Q_m^{(\al,\be)}(t)|
\le (2m+\al+\be+1)^{m+1+\frac{\al+1}{2}}\,x^{\be/2}.
\end{align*}
As $\be$ tends to $+\infty$,
the last expression tends to $0$.
\end{proof}

The following reasoning is similar to
\cite[Lemma~7.2.3]{Vasilevski2008book},
but involves a more complicated expression
with $\jac_j^{(\al,\xi)}$
and uses~Lemma~\ref{lem:jac_tends_to_zero}.

\begin{prop}
\label{prop:a_has_limit_implies_gamma_has_limit}
Let $a\in L^\infty([0,1))$,
$\om\in\bC$,
and $\lim\limits_{r\to 1}a(r)=\om$.
Then
\[
\lim_{\xi\to\infty}\ga_{n,\al}(a)_\xi
= \om I_n.
\]
\end{prop}

\begin{proof}
1. We start with the particular case $\om=0$.
For every $x$ in $(0,1)$ and $\xi$ in $\bNz$, we apply~\eqref{eq:gamma_as_integral_jac} and divide the integral into two parts:
\begin{align*}
|\beta_{a,\al,\xi,j,k}|
&\le \int_0^x |a(\sqrt{t})|\,
\bigl|\jac_j^{(\al,\xi)}(t)\bigr|\,
\bigl|\jac_k^{(\al,\xi)}(t)\bigr|\,
\dif{}t
+\int_x^1 |a(\sqrt{t})|\,
\bigl|\jac_j^{(\al,\xi)}(t)\bigr|\,
\bigl|\jac_k^{(\al,\xi)}(t)\bigr|\,
\dif{}t
\\
&\le x \|a\|_\infty
\,
\left(\sup_{0\le t\le x}
\bigl|\jac_j^{(\al,\xi)}(t)\bigr|\right)
\left(\sup_{0\le t\le x}
\bigl|\jac_k^{(\al,\xi)}(t)\bigr|\right)
+ \sup_{x\le t<1}|a(\sqrt{t})|.
\end{align*}
Let $\eps>0$.
Using the assumption that $a(r)\to0$
as $r\to 1$, we choose $x$ such that the second summand is less than $\eps/2$.
Then, applying Lemma~\ref{lem:jac_tends_to_zero}
with this fixed $x$,
we make the first summand less than $\eps/2$ for $\xi$ large enough.

2. For general $\om$ in $\bC$,
we rewrite $a$ in the form
$b+\om 1_{[0,1)}$,
where $b=a-\om 1_{[0,1)}$ and
\[
\lim_{t\to 1}b(t)=0.
\]
By linearity of integration
and the orthonormal property of the functions~$\jac_j^{(\al,\xi)}$,
\begin{align*}
\be_{a,\al,\xi,j,k}
&=\be_{b,\al,\xi,j,k}
+\om \be_{1_{[0,1)},\al,\xi,j,k}
\\
&=\be_{b,\al,\xi,j,k}
+\om \int_0^1 \jac_j^{(\al,\xi)}(t)
\jac_k^{(\al,\xi)}(t)\,\dif{}t
=\be_{b,\al,\xi,j,k}+\om \de_{j,k}.
\end{align*}
Applying the first part of this proof to $b$ we get 
$\be_{a,\al,\xi,j,k}\to\om\,\de_{j,k}$
as $\xi\to\infty$.
\end{proof}

\section{Algebra of matrix sequences
with scalar limits}
\label{sec:algebra_of_matrix_sequences_with_scalar_limits}

We denote by $\cL_n$ the set of all matrix sequences belonging to $\cS_n$ and having scalar limits, see~\eqref{eq:cL_def}.
It is easily seen that $\cL_n$ is a C*-subalgebra of $\cS_n$.

\begin{prop}
\label{prop:cX_subseteq_cL}
$\cX_{n,\al}\subseteq\cL_n$.
\end{prop}

\begin{proof}
Proposition~\ref{prop:a_has_limit_implies_gamma_has_limit} implies that $\cG_{n,\al}\subseteq\cL_n$.
Since $\cL_n$ is a C*-algebra,
$\cX_{n,\al}\subseteq\cL_n$.
\end{proof}

Our goal is to prove that $\cX_{n,\al}$ coincides with $\cL_n$.
We will use the concept of pure states (see, e.g., Dixmier~\cite[Section~2.5]{Dixmier1977} or Murphy~\cite[Section~5.1]{Murphy1990})
and a non-commutative analog of Stone--Weierstrass theorem proved by Kaplansky.

In this paper, we deal with unital C*-algebras only.
Given a unital C*-algebra $\cU$, we denote by $\PureStates(\cU)$ the set of all pure states of $\cU$.
It is well-known 
(see, e.g.,
~\cite[Theorem~2.9.5]{Dixmier1977} or \cite[Theorem~5.3.5]{Murphy1990}) that the pure states of $\cU$ bijectively correspond to the maximal left ideals of $\cU$
(and to the maximal right ideals of $\cU$).

The next theorem is a particular case of Kaplansky~\cite[Theorem~7.2]{Kaplansky1951}
reformulated as in~\cite[Theorem~11.1.8]{Dixmier1977}.
For the concepts of liminal C*-algebras and C*-algebras of type I, see~\cite[Chapters~4, 9]{Dixmier1977}.

\begin{thm}
\label{thm:Kaplansky}
Let $\cU$ be a unital C*-algebra of type I
and $\cY$ be a C*-subalgebra of $\cU$ separating the pure states of $\cU$.
Then $\cY=\cU$.
\end{thm}

Here the phrase ``$\cY$ separates the pure states of $\cU$'' means that for each pair $\ph,\psi$ in $\PureStates(\cU)$, if $\ph\ne\psi$,
then there exists $Y\in\cY$ such that
$\ph(Y)\ne\psi(Y)$.

In the rest of this section,
we explain that $\cL_n$ is a C*-algebra of type I and find $\PureStates(\cL_n)$.

First, let us recall the explicit description of the pure states of $\Mat_m$, where $m\in\bN$.
It is a particular case of the well-known description of the pure states 
of the C*-algebra of compact operators acting in a Hilbert space
(see \cite[Corollary~4.1.4]{Dixmier1977} or \cite[Example~5.1.1]{Murphy1990}).
We denote by $\bS_m$ the unit sphere in $\bC^m$:
\[
\bS_m \eqdef \{u\in\bC^m\colon\ \|u\|=1\}.
\]

\begin{prop}
\label{prop:pure_states_of_Mat}
$\PureStates(\Mat_m)
=\bigl\{\ph_u\colon\ u\in\bS_m\bigr\}$, where $\ph_u\colon\Mat_m\to\bC$ is defined by 
$\ph_u(A)\eqdef \langle Au,u\rangle$.
\end{prop}

Two vectors $u,v\in\bS_n$ yield the same pure state of $\Mat_m$ if and only if $u$ and $v$ are linearly dependent, i.e., there exists $\tau$ in $\bT$ such that $v=\tau u$.

Let $\hatOm_n\eqdef\Om_n\cup\{\infty\}$
be the one-point (Alexandroff) compactification of $\Om_n$,
where $\Om_n$ is defined by~\eqref{eq:Om_def}.
For every $\xi$ in $\hatOm_n$, we define
\[
d_{n,\xi}
\eqdef
\begin{cases}
\min\{n+\xi,n\},
& \xi\in\Om_n,
\\
1, & \xi=\infty;
\end{cases}
\qquad
\cA_\xi \eqdef \Mat_{d_{n,\xi}}.
\]
In other words,
$\cA_\xi=\Mat_{\min\{n+\xi,n\}}$ 
for $\xi$ in $\Om_n$
and $\cA_\infty=\Mat_1=\bC$.
We define
\[
\hatL_n\eqdef
\left\{A\in\bigoplus_{\xi\in\hatOm_n}\cA_\xi\colon\quad
\lim_{\xi\to\infty}A_\xi = A_\infty I_n
\right\}.
\]
Obviously, $\hatL_n$ is a C*-algebra, and it is isomorphic to $\cL_n$.
Indeed, each $A$ in $\cL_n$ can be extended to the domain $\hatOm_n$ by the rule
$A_\infty\eqdef\om$,
where
$\lim_{\xi\to\infty}A_\xi=\om I_n$.

\begin{prop}
\label{prop:extended_Ln_is_a_Cstar_bundle}
$\hatL_n$ is a C*-subalgebra of
$\bigoplus_{\xi\in\hatOm_n}\cA_\xi$.
Moreover, $\hatL_n$ has the following properties.
\begin{enumerate}
\item
For every $A$ in $\hatL_n$,
the function $\xi\mapsto\|A_\xi\|$ is continuous on $\hatOm_n$.
\item
For every $\xi$ in $\hatOm_n$,
$\{A_\xi\colon\ A\in\hatL_n\}=\cA_\xi$.

\item
$\hatOm_n$ contains the identity of $\bigoplus_{\xi\in\hatOm_n}\cA_\xi$, i.e., the family of identity matrices
$(I_{d_{n,\xi}})_{\xi\in\hatOm_n}$.

\item
For every $f$ in $C(\hatOm_n)$,
the family $(f(\xi)I_{d_{n,\xi}})_{\xi\in\hatOm_n}$
belongs to $\hatL_n$.

\item
For every $A$ in
$\bigoplus_{\xi\in\hatOm_n}\cA_\xi$,
if for every $\eps>0$ and every $\eta$ in $\hatOm_n$
there exists $B$ in $\hatL_n$ and an open neighborhood $N$ of $\eta$ such that $\|A_\xi-B_\xi\|<\eps$ for every $\xi$ in $N$,
then $A\in\hatL_n$.
\end{enumerate}
\end{prop}

\begin{proof}
The fact that $\hatL_n$ is closed in
$\bigoplus_{\xi\in\hatOm_n}\cA_\xi$ can be proved in the same manner as the fact that the space of converging sequences $c(\bNz)$ is closed in $\ell^\infty(\bNz)$.
Properties 1--4 are obvious.

Let $A$ be as in property 5.
Given $\eps>0$, for every $\eta$ in $\hatOm_n$ we choose $B_\eta$ in $\hatL_n$ and an open neighborhood $N_\xi$ of $\hatOm_n$ such that $\|A_\xi-(B_\eta)_\xi\|<\eps$ for every $\xi$ in $N$.
Without loss of generality, we will suppose that $N_\infty$ is of the form $\{\xi\in\bN\colon\ \xi>k\}\cup\{\infty\}$ for some $k$ in $\bN$.
For every $\eta$ in $\Om_n$,
the indicator functions $1_{\{\eta\}}$ are continuous on $\hatOm_n$.
The indicator function
$1_{N_\infty}$ is also continuous on $\hatOm_n$.
Define $X\in\hatL_n$,
\[
X \eqdef \sum_{\eta=-n+1}^k 1_{\{\eta\}} B_\eta
+1_{N_\infty} B_\infty.
\]
Then $\|X-A\|<\eps$.
We conclude that $A\in\hatL_n$
because $\hatL_n$ is complete.
\end{proof}

\begin{prop}
\label{prop:Ln_is_liminal}
$\cL_n$ is a liminal C*-algebra.
Consequently, it is a C*-algebra of type I.
\end{prop}

\begin{proof}
By Proposition~\ref{prop:extended_Ln_is_a_Cstar_bundle},
$((\cA_\xi)_{\xi\in\hatOm_n},\widehat{\cL}_n)$ is a continuous field of C*-algebras
over $\hatOm_n$,
in the sense of Dixmier~\cite[Chapter~10]{Dixmier1977}.
Moreover, since $\hatOm_n$ is a compact space,
$\widehat{\cL}_n$ coincides with the C*-algebra defined by this continuous field of C*-algebras.
Each of the algebras $\cA_\xi$ is finite-dimensional.
Therefore, by~\cite[Corollary~10.4.5]{Dixmier1977},
$\widehat{\cL}_n$ is liminal.
Finally, $\cL_n$ is isomorphic to $\widehat{\cL}_n$.
\end{proof}

\begin{defn}
\label{defn:purestates}
For every $\xi$ in $\Om_n$
and every $u$ in $\bS_{\min\{n,n+\xi\}}$, we define
$\purestate_{\xi,u}\colon \cL_n\to\bC$ by
\[
\purestate_{\xi,u}(A)
\eqdef \langle A_\xi u, u \rangle.
\]
Furthermore, we define
$\purestate_\infty\colon
\cL_n\to\bC$ by
$\purestate_\infty(A)\eqdef \omega$,
if $\lim_{\xi\to\infty}A_\xi=\omega I_n$.
\end{defn}

If $\xi\in\Om_n$
and $u,v\in\bS_{\min\{n,n+\xi\}}$, then
\[
\purestate_{\xi,u}=\purestate_{\xi,v}
\qquad\Longleftrightarrow\qquad
\exists\tau\in\bT\quad v=\tau u.
\]
So, in the definition of the functionals
$\purestate_{\xi,u}$
we could use elements of the projective space
$\mathbb{P}(\bC^{\min\{n,n+\xi\}})$
instead of the unitary vectors,
but we decided to simplify the notation.
Notice also that if $u\in\bS_n$, then
\[
\purestate_\infty(A)
=\lim_{\xi\to\infty} \langle A_\xi u, u \rangle.
\]

\begin{prop}
\label{prop:pure_states_of_Ln}
All functionals from Definition~\ref{defn:purestates}
are pure states of $\cL_n$,
and every pure state of $\cL_n$ has such a form.
\end{prop}

\begin{proof}
By 
Proposition~\ref{prop:extended_Ln_is_a_Cstar_bundle}, $\widehat{\cL}_n$ satisfies conditions from
Kaplansky~\cite[Theorem 3.1 and Corollary 1]{Kaplansky1951}.
See also similar results in Naimark~\cite[Section~26]{Naimark1972}
and Dixmier~\cite[Theorem~10.4.3]{Dixmier1977}.
Thereby, the pure states of $\widehat{\cL}_n$ can be described in terms of the pure states of the ``fibers'' $\cA_\xi$:
\[
\PureStates(\widehat{\cL}_n)
=\Bigl\{\psi\colon
\widehat{\cL}_n\to\bC\colon
\quad
\exists \xi\in\hatOm_n
\quad
\exists \ph\in\PureStates(\cA_\xi)
\quad
\forall A\in\widehat{\cL}_n
\quad
\psi(A)=\ph(A_\xi)\Bigr\}.
\]
We put
$d_\xi\eqdef\min\{n+\xi,n\}$ for $\xi$ in $\Om_n$
and $d_\infty\eqdef 1$.
For every $\xi$ in $\hatOm_n$,
the pure states of $\cA_\xi$
are described by Proposition~\ref{prop:pure_states_of_Mat}.
So,
\[
\PureStates(\widehat{\cL}_n)
=\Bigl\{\psi\colon
\widehat{\cL}_n\to\bC\colon
\quad
\exists \xi\in\hatOm_n
\quad
\exists u\in\bS_{d_\xi}
\quad
\forall A\in\widehat{\cL}_n
\quad
\psi(A)=\langle A_\xi u,u\rangle\Bigr\}.
\]
We recall the natural isomorphism between $\cL_n$ and $\widehat{\cL}_n$,
and obtain the result.
\end{proof}

\section{One generating set of matrices}
\label{sec:one_generating_set_of_matrices}

The idea of this section was proposed
by Ram\'{i}rez~Ortega, Ram\'{i}rez~Mora, and S\'{a}nchez~Nungaray~\cite[end of the proof of Lemma~4.4]{RamirezRamirezSanchez2019},
but very briefly and only for a particular situation.
We explain this idea in a very formal and detailed way because it plays a crucial role in Sections~\ref{sec:algebra_gammas_fixed_frequency} and~\ref{sec:separate_pure_states_associated_to_different_frequencies}.

In this section, $n$ a general natural number,
not necessarily the same number as in the rest of the paper.

Given $p,q$ in $\{0,1,\ldots,n-1\}$,
we denote by $\E_{p,q}$ the basic matrix
\[
\E_{p,q}\eqdef
\bigl[\de_{p,j}\de_{q,k}\bigr]_{j,k=0}^{n-1}.
\]
The order of the matrix $\E_{p,q}$ will be clear from the context.

Before working with a general $n$ in $\bN$,
let us consider an example corresponding to $n=3$.

\begin{example}
Suppose that $G_0,G_1,G_2$ are matrices belonging to $\Mat_3$ of the following form (we denote by $\ast$ arbitrary numbers and by $\bullet$ nonzero numbers):
\[
G_0
=
\matr{ccc}{
\ast & \ast & \ast
\\
\ast & \ast & \ast
\\
\bullet & \ast & \ast
},
\qquad
G_1
=
\matr{ccc}{
\ast & \ast & \ast
\\
\ast & \ast & \ast
\\
0 & \bullet & \ast
},
\qquad
G_2
=
\matr{ccc}{
0 & 0 & 0
\\
0 & 0 & 0
\\
0 & 0 & \bullet
}.
\]
Let us show how to obtain $\E_{2,0}$, $\E_{2,1}$, $\E_{2,2}$
from $G_0$, $G_1$, $G_2$.
For brevity, we will denote by $G_{p,j,k}$ the $(j,k)$th component of $G_p$.
First,
\[
G_2 = G_{2,2,2} \E_{2,2}.
\]
Then
\begin{equation}
\label{eq:E22_via_G}
\E_{2,2}=\frac{1}{G_{2,2,2}} G_2.
\end{equation}
We also put
\[
\nu_{2,2}\eqdef\frac{1}{G_{2,2,2}^2}
\]
and get
\[
\E_{2,2}=\nu_{2,2} G_2^2.
\]
Now consider the product $\E_{2,2} G_1$:
\[
\E_{2,2} G_1
=
\matr{ccc}{
0 & 0 & 0 \\
0 & 0 & 0 \\
0 & G_{1,2,1} & G_{1,2,2}
}
=G_{1,2,1} \E_{2,1} + G_{1,2,2} \E_{2,2}.
\]
Solve this equation for $E_{2,1}$:
\[
\E_{2,1}
=\frac{1}{G_{1,2,1}}
(\E_{2,2}G_1 - G_{1,2,2} \E_{2,2}).
\]
Taking into account~\eqref{eq:E22_via_G},
\begin{equation}
\label{eq:E21_via_G}
\E_{2,1}
=\frac{1}{G_{1,2,1}G_{2,2,2}}\,G_2 G_1
-\frac{G_{1,2,2}\nu_{2,2}}{G_{1,2,1}}\,G_2^2
=\nu_{1,1} G_2 G_1+\nu_{1,2} G_2^2,
\end{equation}
where 
\[
\nu_{1,1}\eqdef \frac{1}{G_{1,2,1}G_{2,2,2}}, \qquad \nu_{1,2}\eqdef -\frac{G_{1,2,2}\nu_{2,2}}{G_{1,2,1}}.
\]
Consider the product $\E_{2,2} G_0$:
\[
\E_{2,2} G_0
=
\matr{ccc}{
0 & 0 & 0 \\
0 & 0 & 0 \\
G_{0,2,0} & G_{0,2,1} & G_{0,2,2}
}
= G_{0,2,0} \E_{2,0}
+ G_{0,2,1} \E_{2,1}
+ G_{0,2,2} \E_{2,2}.
\]
Solve this equation for $\E_{2,0}$:
\[
\E_{2,0}
=\frac{1}{G_{0,2,0}}
(\E_{2,2}G_0 - G_{0,2,1} \E_{2,1} - G_{0,2,2} \E_{2,2})
\]
Taking into account~\eqref{eq:E22_via_G} and~\eqref{eq:E21_via_G},
\begin{align*}
\E_{2,0}
&=\frac{1}{G_{0,2,0}}
\left(\frac{1}{G_{2,2,2}} G_2 G_0 
- G_{0,2,1}(\nu_{1,1}G_2 G_1 + \nu_{1,2}G_2^2)
- G_{0,2,2}\nu_{2,2}G_2^2\right)
\\
&=\nu_{0,0}G_2 G_0
+\nu_{0,1}G_2 G_1
+\nu_{0,2}G_2^2,
\end{align*}
where 
\[
\nu_{0,0}\eqdef
\frac{1}{G_{0,2,0} G_{2,2,2}},\qquad
\nu_{0,1}\eqdef
-\frac{\nu_{1,1} G_{0,2,1}}{G_{0,2,0}},\qquad 
\nu_{0,2}
\eqdef
-\frac{\nu_{1,2} G_{0,2,1} + \nu_{2,2} G_{0,2,2}}{G_{0,2,0}}.
\]
\end{example}

\begin{lem}[construct basic matrices from some special generators]
\label{lem:construct_basic_matrices_from_generators}
Let $G_0,\ldots,G_{n-1}$
be some matrices belonging to $\Mat_n$,
with the following properties:
\begin{enumerate}
\item[1)] $G_{n-1}$ is a nonzero multiple of $\E_{n-1,n-1}$,
\item[2)] for every $p$ in $\{0,\ldots,n-2\}$,
$(G_p)_{n-1,p}\ne0$,
\item[3)] for every $p$ in $\{1,\ldots,n-2\}$
and every $k<p$,
$(G_p)_{n-1,k}=0$.
\end{enumerate}
Then for every $p$ in $\{0,\ldots,n-1\}$,
there exist $\nu_{p,p},\ldots,\nu_{p,n-1}\in\bC$ such that
\begin{equation}
\label{eq:E_last_row_via_G}
\E_{n-1,p}
= G_{n-1} \sum_{j=p}^{n-1} \nu_{p,j} G_j.
\end{equation}
Moreover, for every $p,q$ in $\{0,\ldots,n-1\}$,
\begin{equation}
\label{eq:all_E_via_G}
\E_{p,q}
=
\left(\sum_{j=p}^{n-1} \overline{\nu_{p,j}} G_j^\ast\right) G_{n-1}^\ast G_{n-1}
\left(\sum_{k=q}^{n-1} \nu_{q,k} G_k\right).
\end{equation}
\end{lem}

\begin{proof}
We denote by $G_{p,j,k}$ the $(j,k)$th component of $G_p$.
Assumption 1) means that
\[
G_{n-1}=G_{n-1,n-1,n-1} \E_{n-1,n-1},\qquad
G_{n-1,n-1,n-1}\ne0.
\]
Therefore,
\begin{equation}
\label{eq:E_last_via_G_last_1}
\E_{n-1,n-1}
=\frac{1}{G_{n-1,n-1,n-1}}G_{n-1}.
\end{equation}
We also put $\nu_{n-1,n-1}=1/G_{n-1,n-1,n-1}^{2}$
and obtain~\eqref{eq:E_last_row_via_G} with $p=n-1$:
\[
\E_{n-1,n-1}
=\nu_{n-1,n-1} G_{n-1}^2.
\]
Now we proceed by steps, enumerating them in descending order from $n-2$ to $0$.
After performing the steps $n-2,\ldots,p+1$,
we already know that
$\E_{n-1,p+1},\ldots,\E_{n-1,n-1}$ are linear combinations of the products $G_{n-1} G_j$, $j\in\{p+1,\ldots,n-1\}$:
\begin{equation}
\label{eq:E_greater_than_p}
\E_{n-1,q}
= G_{n-1}
\sum_{j=q}^{n-1}\nu_{q,j} G_j
\qquad(p+1\le q<n).
\end{equation}
Step $p$.
In the matrix $\E_{n-1,n-1} G_p$,
the first $n-1$ rows are zero,
and the $(n-1)$th row coincides with the $(n-1)$th row of $G_p$.
Moreover, by assumption 3),
the first $p-1$ components in this row are zero.
Therefore,
\[
\E_{n-1,n-1} G_p
=G_{p,n-1,p} \E_{n-1,p}
+\sum_{q=p+1}^{n-1} G_{p,n-1,q} \E_{n-1,q}.
\]
By assumption 2), $G_{p,n-1,p}\ne0$.
Hence, we can solve the above equation for $\E_{n-1,p}$:
\[
\E_{n-1,p}
=\frac{1}{G_{p,n-1,p}}
\left(
\E_{n-1,n-1} G_p
-\sum_{q=p+1}^{n-1} G_{p,n-1,q} \E_{n-1,q}
\right).
\]
Substitute~\eqref{eq:E_last_via_G_last_1} and~\eqref{eq:E_greater_than_p}:
\begin{align*}
\E_{n-1,p}
&=\frac{1}{G_{p,n-1,p}}
\left(\frac{1}{G_{n-1,n-1,n-1}} G_{n-1} G_p
-\sum_{q=p+1}^{n-1}
G_{n-1} \sum_{j=q}^{n-1} \nu_{q,j}G_{p,n-1,q}  G_j\right)
\\
&=G_{n-1}
\left(\frac{1}{G_{p,n-1,p}G_{n-1,n-1,n-1}}
G_p
-\sum_{j=p+1}^{n-1}
\left(\sum_{q=p+1}^{j}
\frac{\nu_{q,j}G_{p,n-1,q}}{G_{p,n-1,p}}\right) G_j
\right).
\end{align*}
So, we have obtained~\eqref{eq:E_last_row_via_G} with
\[
\nu_{p,p}
=\frac{1}{G_{p,n-1,p}G_{n-1,n-1,n-1}},\qquad
\nu_{p,j}
=-\frac{1}{G_{p,n-1,p}}\sum_{q=p+1}^j \nu_{q,j}\,G_{p,n-1,q}
\quad(p+1\le j<n).
\]
After performing all steps ($n-2,n-1,\ldots,1,0$),
we have proved~\eqref{eq:E_last_row_via_G} for all $p$.
Finally, since
\[
\E_{p,q}
=\E_{p,n-1} \E_{n-1,q}
=\E_{n-1,p}^\ast \E_{n-1,q},
\]
\eqref{eq:all_E_via_G} follows from~\eqref{eq:E_last_row_via_G}.
\end{proof}

\begin{remark}
We have tested formulas~\eqref{eq:E_last_row_via_G} and~\eqref{eq:all_E_via_G} from Lemma~\ref{lem:construct_basic_matrices_from_generators} numerically in SageMath, using pseudorandom matrices $G_j$ with the required structure, for $n=2,3,4$.
\end{remark}

\begin{prop}
\label{prop:algebra_generated_by_matrices}
Let $G_0,\ldots,G_{n-1}$
be some matrices belonging to $\Mat_n$,
with the following properties:
\begin{enumerate}
\item[1)] $G_{n-1}$ is a nonzero multiple of $\E_{n-1,n-1}$,
\item[2)] for every $p$ in $\{0,\ldots,n-2\}$,
$(G_p)_{n-1,p}\ne0$,
\item[3)] for every $p$ in $\{1,\ldots,n-2\}$
and every $k<p$,
$(G_p)_{n-1,k}=0$.
\end{enumerate}
Let $\cA$ be the C*-subalgebra of $\Mat_n$
generated by $G_0,\ldots,G_{n-1}$.
Then $\cA=\Mat_n$.
\end{prop}

\begin{proof}
By Lemma~\ref{lem:construct_basic_matrices_from_generators},
$\E_{p,q}\in\cA$
for every $p,q$ in $\{0,\ldots,n-1\}$.
The linear combinations of
$\E_{p,q}$, $p,q\in\{0,\ldots,n-1\}$,
yield the whole algebra $\cA$.
\end{proof}

\begin{remark}
\label{rem:construct_basic_matrices_from_real_symmetric_generators}
Let $G_0,\ldots,G_{n-1}$
be some matrices belonging to $\Mat_n(\bC)$
and satisfying the conditions
from Lemma~\ref{lem:construct_basic_matrices_from_generators}.
Moreover, suppose that the matrices $G_0,\ldots,G_{n-1}$ are real and symmetric.
Then the coefficients $\nu_{p,j}$ are real and~\eqref{eq:all_E_via_G} simplifies to 
\[
\E_{p,q}
=
\left(\sum_{j=p}^{n-1} \nu_{p,j} G_j\right) G_{n-1}^2
\left(\sum_{k=q}^{n-1} \nu_{q,k} G_k\right).
\]
\end{remark}

In this paper, it will be convenient to use the following technical concept.
A more precise name would be ``$p$-antitriangular matrix where all components of the $p$th antidiagonal are nonzero'',
but we will use a simplified terminology.

\begin{defn}[$p$-antitriangular matrix]
\label{def:antitriangular}
Let $A\in\Mat_n$ and $p\in\bZ$.
We say that $A$
is
\emph{$p$-antitriangular},
if for every $j,k$ in $\{0,\ldots,n-1\}$,
\begin{align*}
A_{j,k}&=0,\quad\text{if}\quad j+k<p;
\\
A_{j,k}&\ne0,\quad
\text{if}\quad
j+k=p.
\end{align*}
\end{defn}

\begin{remark}
\label{rem:antitriangular_is_zero}
If $A\in\Mat_n$ and $A$ is $p$-antitriangular with $p>2n-2$, then $A$ is the zero matrix.
Indeed, 
$j+k\le 2n-2<p$
for every $j,k$ in $\{0,\ldots,n-1\}$.
\end{remark}

\begin{remark}
\label{rem:antitriangular_serve_as_generators}
If $G_0,\ldots,G_{n-1}\in\Mat_n$ such that for every $p$ the matrix $G_p$ is $(n-1+p)$-antitriangular,
then $G_0,\ldots,G_{n-1}$ satisfy the conditions of Lemma~\ref{lem:construct_basic_matrices_from_generators}.
\end{remark}

\section{Algebra generated by the matrices corresponding to a fixed frequency}
\label{sec:algebra_gammas_fixed_frequency}

The following generating symbols play a crucial role in this paper.

\begin{defn}
\label{def:g_p}
For every $p$ in $\bNz$,
we define $g_{p,\al}\in L^\infty([0,1))$ by
\begin{equation}
\label{eq:gen_symbol_jac}
g_{p,\al}(t)\eqdef
Q^{(\al,0)}_p(t^2).
\end{equation}
For brevity, we write $g_p$ instead of $g_{p,\al}$, because $\al$ is fixed.
\end{defn}

We notice that $g_p$ is a real function and $g_p\in\BL$.

From now on, for brevity,
we write $\ga$ instead of $\ga_{n,\al}$, because $n$ and $\al$ are fixed.

\begin{lem}
\label{lem:gamma_xi_associated_to_Q}
Let $p\in\bNz$ and $\xi\in\Om_n$.
Then $\ga(g_p)_\xi$ is $(p-|\xi|)$-antitriangular,
real, and symmetric.
\end{lem}

\begin{proof}
By Remark~\ref{rem:gammas_are_symmetric}, $\ga(g_p)_\xi$ is symmetric.
Since $g_p$ is a real function,
all components of $\ga(g_p)_\xi$ are real numbers.

We are left to prove that $\ga(g_p)$ is $(p-|\xi|)$-antitriangular,
in the sense of Definition~\ref{def:antitriangular}.
First, consider the case if $\xi\ge0$.
Let $j,k\in\{0,\ldots,n-1\}$.
Define $f\colon[0,1)\to\bR$,
$f(t)\eqdef Q_j^{(\al,\xi)}(t)Q_k^{(\al,\xi)}(t)\,t^\xi$.
Then
\begin{align*}
\ga(g_p)_{\xi,j,k}
&=
\jaccoef{\al}{\xi}{j}
\jaccoef{\al}{\xi}{k}
\int_0^1 Q_p^{(\al,0)}(t)\,
Q_j^{(\al,\xi)}(t)
Q_k^{(\al,\xi)}(t)
\,
(1-t)^\al t^\xi\,\dif{}t
\\[1ex]
&=
\jaccoef{\al}{\xi}{j}
\jaccoef{\al}{\xi}{k}
\int_0^1
Q_p^{(\al,0)}(t)\,
f(t)\,
(1-t)^\al\,\dif{}t.
\end{align*}
We notice that $f$ is a polynomial of degree $j+k+\xi$.
Therefore, by Proposition~\ref{prop:int_Q_by_polynomial},
\begin{align*}
\ga(g_p)_{\xi,j,k}
&=0,\quad \text{if}\quad
j+k<p-\xi,
\\
\ga(g_p)_{\xi,j,k}
&\ne 0,\quad \text{if}\quad
j+k=p-\xi.
\end{align*}
This means that $\ga(g_p)$ is $(p-\xi)$-antitriangular.

For $\xi$ in $\{-n+1,\ldots,-1\}$,
the reasoning is very similar, but
$j,k\in\{0,\ldots,n+\xi-1\}$,
and $f$ is defined by
\[
f(t)\eqdef Q_j^{(\al,|\xi|)}(t)Q_k^{(\al,|\xi|)}(t)t^{|\xi|}.
\]
Then $f$ is  a polynomial of degree $j+k+|\xi|$,
\[
\ga(g_p)_{\xi,j,k}
=\be_{g_p,\al,\xi,j,k}
=\be_{g_p,\al,|\xi|,j,k}
=\jaccoef{\al}{|\xi|}{j}
\jaccoef{\al}{|\xi|}{k}
\int_0^1
Q_p^{(\al,0)}(t)\,f(t)\,
(1-t)^\al\,\dif{}t,
\]
and by Proposition~\ref{prop:int_Q_by_polynomial} we obtain that $\ga(g_p)$ is $(p-|\xi|)$-antitriangular.
\end{proof}

\begin{prop}
\label{prop:algebra_generated_by_gammas_for_a_fixed_positive_frequency}
Let $\xi\in\bNz$.
Then the C*-algebra generated by the matrices
\[
\ga(g_{n-1+\xi})_\xi,\quad
\ga(g_{n+\xi})_\xi,\quad
\ldots,\quad
\ga(g_{2n-2+\xi})_\xi,
\]
coincides with $\Mat_n$.
\end{prop}

\begin{proof}
For each $p$ in $\{0,\ldots,n-1\}$, we put
\[
G_p\eqdef\ga(g_{n-1+\xi+p})_\xi.
\]
By Lemma~\ref{lem:gamma_xi_associated_to_Q},
$G_p$ is $(n-1+p)$-antitriangular.
By Remark~\ref{rem:antitriangular_serve_as_generators},
matrices $G_0,\ldots,G_{n-1}$ satisfy the conditions required in Proposition~\ref{prop:algebra_generated_by_matrices},
and the C*-algebra generated by
these matrices is $\Mat_n$.
\end{proof}

For $\xi<0$,
we have the following analog
of Proposition~\ref{prop:algebra_generated_by_gammas_for_a_fixed_positive_frequency}.

\begin{prop}
\label{prop:algebra_generated_by_gammas_for_a_fixed_negative_frequency}
Let $\xi\in\{-n+1,\ldots,-1\}$.
Then the C*-algebra generated by the matrices
\[
\ga(g_{n-1})_\xi,\quad
\ga(g_{n})_\xi,\quad
\ldots,\quad
\ga(g_{2n-2+\xi})_\xi,
\]
coincides with $\Mat_{n+\xi}$.
\end{prop}

\begin{proof}
For every $p$ in $\{0,\ldots,n-1+\xi\}$,
we put
\[
G_p \eqdef \ga(g_{n-1+p})_\xi.
\]
By Lemma~\ref{lem:gamma_xi_associated_to_Q},
$G_p$ is $(n+\xi-1+p)$-antidiagonal.
Therefore, the matrices $G_p$ with
$p$ in $\{0,\ldots,n+\xi-1\}$
satisfy conditions of Proposition~\ref{prop:algebra_generated_by_matrices},
with $n+\xi$ instead of $n$.
\end{proof}

\begin{lem}
\label{lem:linear_independent_unitary_vectors}
Let $u,v\in\bS_n$ such that $u,v$ are linearly independent.
Then there exist $p,q$ in $\{0,\ldots,n-1\}$ such that
\[
u_p \overline{u_q} \ne v_p \overline{v_q}.
\]
\end{lem}

\begin{proof}
Let $p\in\{0,\ldots,n-1\}$ such that $u_p\ne0$.
If $|v_p|\ne|u_p|$, then we obtain the desired result with $q=p$.
Suppose that $|v_p|=|u_p|$.
Put $\tau \eqdef v_p/u_p$.
Then $|\tau|=1$ and $\tau^{-1}=\overline{\tau}$.
Since $v\ne \tau u$, there exists $q$ in $\{0,\ldots,n-1\}$ such that $v_{q}\neq \tau u_{q}$, or $\overline{v_{q}}\neq \overline{\tau}\overline{u_{q}}=\tau^{-1}\overline{u_{q}}$. Thus, multiplying both sides of the last inequality by $v_p$ we get
\[
v_{p}\overline{v_{q}}\neq v_{p}\tau^{-1}\overline{u_{q}}=\tau u_{p}\tau^{-1}\overline{u_{q}}=u_{p}\overline{u_{q}}.\qedhere
\]
\end{proof}

\begin{prop}[separation of the pure states associated to the same frequency]
\label{prop:separation_pure_states_same_frequency}
Let $\xi\in\Om_n$,
$d=\min\{n,n+\xi\}$,
and $u,v\in\bS_d$
such that $u$ and $v$ are linearly independent.
Then there exists $X$ in $\cX_{n,\al}$ such that
\[
\purestate_{\xi,u}(X)\ne\purestate_{\xi,v}(X).
\]
\end{prop}

\begin{proof}
This is corollary from Propositions~\ref{prop:algebra_generated_by_gammas_for_a_fixed_positive_frequency} and~\ref{prop:algebra_generated_by_gammas_for_a_fixed_negative_frequency}, but we will give a more explicit proof.
For every $k$ in $\bN_0$, we put
$A_k\eqdef\ga(g_k)$.
For every $p$ in $\{0,\ldots,d-1\}$, let
\[
G_p \eqdef (A_{d-1+|\xi|+p})_\xi.
\]
By Lemma~\ref{lem:gamma_xi_associated_to_Q},
$G_p$ is $(d-1+p)$-antidiagonal.
So, the matrices $G_0,\ldots,G_{d-1}$
satisfy the assumptions of Lemma~\ref{lem:construct_basic_matrices_from_generators}, with $d$ instead of $n$.

Using Lemma~\ref{lem:linear_independent_unitary_vectors} we find $p,q$ in $\{0,\ldots,d-1\}$ such that
$u_p \overline{u_q}\ne v_p \overline{v_q}$.
By Lemma~\ref{lem:construct_basic_matrices_from_generators}, there exist coefficients $\nu_{p,j}$ and $\nu_{q,j}$ such that
\[
\E_{p,q}
=
\left(\sum_{j=p}^{d-1}\nu_{p,j}G_j\right)
G_{d-1}^2
\left(\sum_{j=q}^{d-1}\nu_{q,j}G_j\right).
\]
Put
\[
X
\eqdef
\left(\sum_{j=p}^{d-1}
\nu_{p,j}A_{d-1+|\xi|+j}\right)
A_{2d+|\xi|-2}^2
\left(\sum_{j=q}^{d-1}
\nu_{q,j}A_{d-1+|\xi|+j}\right).
\]
Then $X\in\cX_{n,\al}$,
$X_\xi = \E_{p,q}$, and
\[
\purestate_{\xi,u}(X)
=\langle \E_{p,q}u,u\rangle
=u_p \overline{u_q}
\ne
v_p \overline{v_q}
=\langle \E_{p,q}v,v\rangle
=\purestate_{\xi,v}(X).
\qedhere
\]
\end{proof}

\begin{remark}
Instead of the generating symbols~\eqref{eq:gen_symbol_jac},
we could use generating symbols of the form
$g_{p,\xi}(t)\eqdef Q_p^{(\al,\xi)}(t^2)$,
but the notation would be more complicated,
especially in Sections~\ref{sec:separate_pure_states_associated_to_different_frequencies}
and~\ref{sec:separate_pure_states_associated_to_different_frequencies_when_the_lower_frequency_is_negative}.
\end{remark}

\section{Separation of the limit value from the other pure states}
\label{eq:separate_limit_value}

In this section, we show that $\cG_{n,\al}$
separates $\purestate_\infty$ from the other pure states of $\cL_n$.
The proof is very similar to a reasoning in~\cite[Lemma~4.3]{RamirezSanchez2015}.

\begin{lem}
\label{lem:purestate_via_integral}
Let $\xi\in\Om_n$,
$d=\min\{n+\xi,n\}$,
$u\in\bS_d$,
$a\in\BL$,
and $A\eqdef\ga(a)$.
Then
\begin{equation}
\label{eq:purestate_via_integral}
\purestate_{\xi,u}(A)
=
\int_0^1 a(\sqrt{t})\,|F_{\al,\xi,u}(t)|^2\,\dif{}t,
\end{equation}
where
\begin{equation}
\label{eq:F_from_purestate_via_integral}
F_{\al,\xi,u}(t)\eqdef
\sum_{k=0}^{d-1} u_k
\jac_k^{(\al,|\xi|)}(t).
\end{equation}
\end{lem}

\begin{proof}
By the definitions of $\purestate_{\xi,u}$ and $\ga(a)_\xi$,
\begin{align*}
\purestate_{\xi,u}(A)
&=
\langle \ga(a)_\xi u,u\rangle
=
\sum_{j,k=0}^{d-1}
\overline{u_j}\,u_k
\,\be_{a,\al,\xi,j,k}
=
\sum_{j,k=0}^{d-1}
\overline{u_j}\,
u_k\,
\int_0^1 a(t)\,
\jac_{j}^{(\al,|\xi|)}(t)\,
\jac_{k}^{(\al,|\xi|)}(t)\,
\dif{}t
\\
&=
\int_0^1 a(t)
\left(\,\sum_{j=0}^{d-1} \overline{u_j} \jac_j^{(\al,|\xi|)}(t)
\right)
\left(\,\sum_{k=0}^{d-1} u_k \jac_k^{(\al,|\xi|)}(t)
\right)\,\dif{}t.
\end{align*}
The last expression can be written as~\eqref{eq:purestate_via_integral}.
\end{proof}

\begin{prop}
\label{prop:purestate_indicator_function}
Let $\xi\in\Om_n$,
$d=\min\{n,n+\xi\}$,
and $u\in\bS_d$.
Then there exists $A$ in $\cG_{n,\al}$ such that
\[
\purestate_\infty(A)\ne\purestate_{\xi,u}(A).
\]
\end{prop}

\begin{proof}
Let $s\in(0,1)$, e.g., $s=1/2$.
Consider $A\eqdef\ga_{n,\al}(a)$,
where $a$ is the indicator function of $[0,s)$, i.e.,
$a\eqdef 1_{[0,s)}$.
By Lemma~\ref{lem:purestate_via_integral},
\[
\purestate_{\xi,u}(A)
=\int_0^{s^2}
|F_{\al,\xi,u}(t)|^2\,\dif{}t,
\]
where $F_{\al,\xi,u}$ is defined by~\eqref{eq:F_from_purestate_via_integral}.
Functions $\jac^{(\al,|\xi|)}_j$
with different $j$
are linearly independent.
Therefore, $|F_{\al,\xi,u}|^2$
is strictly positive except
for a finite subset of $[0,1)$,
and $\purestate_{\xi,u}(A)>0$.

On the other hand, $1_{[0,s)}$ vanishes near the point $1$.
Hence, by Proposition~\ref{prop:a_has_limit_implies_gamma_has_limit},
\[
\lim_{\xi\to\infty} \ga(A)_\xi = 0_{n\times n}.
\]
Therefore, $\purestate_\infty(A)=0$.
\end{proof}

\section{Separation of pure states associated to different positive frequencies}
\label{sec:separate_pure_states_associated_to_different_frequencies}

We recall that $g_p$ is introduced in Definition~\ref{def:g_p}.

\begin{lem}
\label{lem:when_gamma_g_is_zero}
Let $p\in\bNz$ and $\xi\in\Om_n$
such that 
$p-|\xi| > 2d-2$,
where
$d=\min\{n+\xi,n\}$.
Then $\ga(g_p)_\xi$ is the zero matrix in $\Mat_d$.
\end{lem}

\begin{proof}
By Lemma~\ref{lem:gamma_xi_associated_to_Q}, $\ga(g_p)_\xi$ is $(p-|\xi|)$-antitriangular.
Moreover,
$p-|\xi|>2d-2$.
By Remark~\ref{rem:antitriangular_is_zero}, $\ga(g_p)_\xi$ is the zero matrix.
\end{proof}

Before proving the general result (Proposition~\ref{prop:separate_pure_states_associated_to_different_positive_frequencies}),
consider a simple particular case.

\begin{example}
Let $n=2$, $\xi\in\bNz$, $\eta=\xi+1$.
Consider the matrix sequences
$A_p\eqdef\ga(g_p)$
for $p=\eta+1$ and $p=\eta+2$.
Due to Lemmas~\ref{lem:gamma_xi_associated_to_Q}
and~\ref{lem:when_gamma_g_is_zero},
the elements $\xi$ and $\eta$
of these matrix sequences
have the following structure:
\[
(A_{\eta+1})_\xi
=
\begin{bmatrix}
0 & 0 \\
0 & \bullet
\end{bmatrix},
\qquad
(A_{\eta+1})_\eta
=
\begin{bmatrix}
0 & \bullet \\
\bullet & \ast
\end{bmatrix},
\]
\[
(A_{\eta+2})_\xi
=
\begin{bmatrix}
0 & 0 \\
0 & 0
\end{bmatrix},
\qquad
(A_{\eta+2})_\eta
=
\begin{bmatrix}
0 & 0 \\
0 & \bullet
\end{bmatrix}.
\]
For example, if $\xi=1$, then
we pay attention to the elements with indices $\xi=1$ and $\eta=2$ in the sequences $A_3$ and $A_4$:
\begin{align*}
A_3=
\Biggl(
\begin{bmatrix}
0
\end{bmatrix},
\quad
\begin{bmatrix}
0 & 0 \\
0 & 0
\end{bmatrix},
\quad
\underbrace{
\begin{bmatrix}
0 & 0 \\
0 & \bullet
\end{bmatrix}}_{(A_3)_1},
\quad
\underbrace{
\begin{bmatrix}
0 & \bullet \\
\bullet & \ast
\end{bmatrix}}_{(A_3)_2},
\quad
\begin{bmatrix}
\bullet & \ast \\
\ast & \ast
\end{bmatrix},
\quad \ldots
\Biggr),
\\[1ex]
A_4=
\Biggl(
\begin{bmatrix}
0
\end{bmatrix},
\quad
\begin{bmatrix}
0 & 0 \\
0 & 0
\end{bmatrix},
\quad
\underbrace{
\begin{bmatrix}
0 & 0 \\
0 & 0
\end{bmatrix}}_{(A_4)_1},
\quad
\underbrace{
\begin{bmatrix}
0 & 0 \\
0 & \bullet
\end{bmatrix}}_{(A_4)_2},
\quad
\begin{bmatrix}
0 & \bullet \\
\bullet & \ast
\end{bmatrix},
\quad \ldots
\Biggr).
\end{align*}
In general, if $\xi\in\bNz$ and $\eta=\xi+1$,
then we work with the matrix sequences $A_{\eta+1}$ and $A_{\eta+2}$.
The idea is to apply some algebraic operations to these matrix sequences
and obtain totally different matrices in the positions $\xi$ and $\eta$.
Let
\[
G_0 \eqdef (A_{\eta+1})_\eta,\qquad
G_1 \eqdef (A_{\eta+2})_\eta.
\]
Notice that $G_0$ is $1$-antitriangular and $G_1$ is $2$-antitriangular.
Therefore,
$G_0$ and $G_1$ satisfy conditions of Lemma~\ref{lem:construct_basic_matrices_from_generators}.
By this lemma,
$\E_{0,0}$ and $\E_{1,1}$ can be written in terms of $G_0$ and $G_1$ as follows (with some coefficients $\nu_{0,0}$, $\nu_{0,1}$, $\nu_{1,1}$):
\begin{align}
\label{eq:E00_from_Aeta_example2}
\E_{0,0}
&=(\nu_{0,0}\,(A_{\eta+1})_\eta
+ \nu_{0,1}\,(A_{\eta+2})_\eta)\,
(A_{\eta+2})_{\eta}^2\,
(\nu_{0,0}\,(A_{\eta+1})_\eta
+ \nu_{0,1}\,(A_{\eta+2})_{\eta}),
\\[1ex]
\label{eq:E11_from_Aeta_example2}
\E_{1,1}
&=\nu_{1,1}\,(A_{\eta+2})_{\eta}^2.
\end{align}
Consider the whole matrix sequences constructed from $A_{\eta+1}$ and $A_{\eta+2}$ by the same rules:
\begin{align*}
X
&\eqdef(\nu_{0,0}\,A_{\eta+1}
+ \nu_{0,1}\,A_{\eta+2})\,
A_{\eta+2}^2\,
(\nu_{0,0}\,A_{\eta+1}
+ \nu_{0,1}\,A_{\eta+2}),
\\[1ex]
Y
&\eqdef \nu_{1,1}\,A_{\eta+2}^2.
\end{align*}
Then, by~\eqref{eq:E00_from_Aeta_example2}
and~\eqref{eq:E11_from_Aeta_example2},
the $\eta$th elements of these matrix sequences are basic diagonal matrices:
\[
X_\eta=\E_{0,0},\qquad
Y_\eta=\E_{1,1}.
\]
On the other hand, since $(A_{\eta+2})_\xi$ is the zero matrix,
the $\xi$th elements of these matrix sequences are zero:
\[
X_\xi=0_{2\times 2},\qquad
Y_\xi=0_{2\times 2}.
\]
Now it is easy to see that the sequences $X$ and $Y$ are sufficient to separate any pair of pure states of the form
$\purestate_{\xi,u}$ and $\purestate_{\eta,v}$,
where $u,v\in\bS_2$.
The condition $\|v\|=1$ implies that $v_0\ne0$ or $v_1\ne0$.

\medskip\noindent
Case 1. Let $v_0\ne0$.
Then
\[
\purestate_{\xi,u}(X)
=\langle X_\xi u,u\rangle = 0,\qquad
\purestate_{\eta,v}(X)
=\langle X_\eta v,v\rangle
=\langle \E_{0,0} v,v\rangle
=|v_0|^2
>0.
\]

\medskip\noindent
Case 2. Let $v_1\ne 0$.
Then
\[
\purestate_{\xi,u}(Y)
=\langle Y_\xi u,u\rangle = 0,\qquad
\purestate_{\eta,v}(Y)
=\langle Y_\eta v,v\rangle
=\langle \E_{1,1}v,v\rangle
=|v_1|^2
>0.
\]
\end{example}

\begin{prop}[separation of pure states associated to different positive frequencies]
\label{prop:separate_pure_states_associated_to_different_positive_frequencies}
Let $\xi,\eta\in\bNz$ such that $\xi<\eta$,
and let $u,v\in\bS_n$.
Then there exists $X$ in $\cX_{n,\al}$ such that
\[
\purestate_{\xi,u}(X)
\ne\purestate_{\eta,v}(X).
\]
\end{prop}

\begin{proof}
For every $k$ in $\bNz$, we put $A_k \eqdef \ga(g_k)$.
Furthermore, for each $q$ in $\{0,\ldots,n-1\}$,
we use the following notation:
\[
G_q \eqdef
(A_{n-1+\eta+q})_\eta.
\]
By Lemma~\ref{lem:gamma_xi_associated_to_Q},
$G_q$ is $(n-1+q)$-antitriangular.
Hence, by Remark~\ref{rem:antitriangular_serve_as_generators},
$G_0,\ldots,G_{n-1}$ satisfy conditions of Lemma~\ref{lem:construct_basic_matrices_from_generators}.
Moreover, these matrices are real and symmetric.

Let $p\in\{0,\ldots,n-1\}$ such that $v_p\ne0$.
By Lemma~\ref{lem:construct_basic_matrices_from_generators}
and Remark~\ref{rem:construct_basic_matrices_from_real_symmetric_generators},
there exist $\nu_{p,p},\ldots,\nu_{p,n-1}$ in $\bR$ such that
\[
\E_{p,p}
=
\left(\,\sum_{j=p}^{n-1} \nu_{p,j}G_j\right)G_{n-1}^2
\left(\,\sum_{j=p}^{n-1} \nu_{p,j}G_j\right),
\]
i.e.,
\begin{equation}
\label{eq:Epp_via_Aeta}
\E_{p,p}
=
\left(\,
\sum_{j=p}^{n-1} \nu_{p,j}
(A_{n-1+\eta+j})_\eta
\right)
(A_{2n-2+\eta})_\eta^2
\left(\,
\sum_{j=p}^{n-1} \nu_{p,j}
(A_{n-1+\eta+j})_\eta
\right).
\end{equation}
We construct $X\in\cX_{n,\al}$ by a similar rule, using the whole sequences $A_{n-1+\eta+q}$ instead of $G_q$:
\begin{equation}
\label{eq:define_X_via_A}
X \eqdef 
\left(\,\sum_{j=p}^{n-1} \nu_{p,j}A_{n-1+\eta+j}\right)
A_{2n-2+\eta}^2
\left(\,\sum_{j=p}^{n-1} \nu_{p,j}A_{n-1+\eta+j}\right).
\end{equation}
Then, by~\eqref{eq:Epp_via_Aeta},
$X_\eta = \E_{p,p}$ and
\[
\purestate_{\eta,v}(X)
=\langle X_\eta v,v\rangle
=\langle \E_{p,p} v, v\rangle
=|v_p|^2 > 0.
\]
On the other hand,
by Lemma~\ref{lem:when_gamma_g_is_zero},
$(A_{2n-2+\eta})_\xi$ is the zero matrix in $\Mat_n$
because
\[
2n-2+\eta-\xi > 2n-2.
\]
From~\eqref{eq:define_X_via_A},
we get 
$X_\xi = 0_{n\times n}$.
Therefore,
$\purestate_{\xi,u}(X)
=\langle X_\xi u,u\rangle
=0$.
\end{proof}

\section{\texorpdfstring{Separation of pure states associated to different $\boldsymbol{\xi}$ and $\boldsymbol{\eta}$,
for $\boldsymbol{\xi<0}$}{Separation of pure states associated to different xi and eta, for xi<0}}
\label{sec:separate_pure_states_associated_to_different_frequencies_when_the_lower_frequency_is_negative}

In this section,
we will separate pure states of the form~$\purestate_{\xi,u}$ and $\purestate_{\eta,v}$,
where $\xi,\eta$ in $\Om_n$, $\xi<\eta$ and $\xi<0$.
There are three possible cases:
\begin{itemize}

\item $-n+1\le \xi<\eta<0$
(Proposition~\ref{prop:separate_pure_states_associated_to_different_negative_frequencies});

\item $-\eta\leq \xi<0$ (Proposition~\ref{prop:separate_pure_states_associated_to_different_frequences_xi_eq_minus_eta});

\item $1\le\eta\le n-2$,
$-n+1\le\xi<-\eta$
(Proposition~\ref{prop:separate_pure_states_associated_to_different_frequences_xi_less_minus_eta}).
\end{itemize}

\begin{example}
\label{example:gammas_negative_2}
Let $n=2$.
Consider the matrix sequences
$A_p\eqdef\ga(g_p)$
for $p\in\{0,1,2,3,4\}$.
Due to Lemma~\ref{lem:gamma_xi_associated_to_Q}
and Remark~\ref{rem:gammas_for_negative_frequencies},
the first elements
of these matrix sequences
have the following structure:
\[
A_1=
\bigleftparenthesis
\
\underbrace{
\begin{bmatrix}
\bullet
\end{bmatrix}
}_{(A_1)_{-1}},
\quad 
\underbrace{
\begin{bmatrix}
0 & \bullet \\
\bullet & \ast
\end{bmatrix}
}_{(A_1)_{0}},
\quad
\underbrace{
\begin{bmatrix}
\bullet & \ast \\
\ast & \ast
\end{bmatrix}
}_{(A_1)_1},
\quad
\underbrace{
\begin{bmatrix}
\ast & \ast \\
\ast & \ast
\end{bmatrix}
}_{(A_1)_2},
\quad
\underbrace{
\begin{bmatrix}
\ast & \ast \\
\ast & \ast
\end{bmatrix}
}_{(A_1)_3},\quad
\ldots
\bigrightparenthesis,
\]

\[
A_2=
\bigleftparenthesis
\
\underbrace{
\begin{bmatrix}
0
\end{bmatrix}
}_{(A_2)_{-1}},
\quad 
\underbrace{
\begin{bmatrix}
0 & 0 \\
0 & \bullet
\end{bmatrix}
}_{(A_2)_{0}},
\quad
\underbrace{
\begin{bmatrix}
0 & \bullet \\
\bullet & \ast
\end{bmatrix}
}_{(A_2)_1},
\quad
\underbrace{
\begin{bmatrix}
\bullet & \ast \\
\ast & \ast
\end{bmatrix}
}_{(A_2)_2},
\quad
\underbrace{
\begin{bmatrix}
\ast & \ast \\
\ast & \ast
\end{bmatrix}
}_{(A_2)_3},\quad
\ldots
\bigrightparenthesis,
\]
\[
A_3=
\bigleftparenthesis
\
\underbrace{
\begin{bmatrix}
0
\end{bmatrix}
}_{(A_3)_{-1}},
\quad
\underbrace{
\begin{bmatrix}
0 & 0 \\
0 & 0
\end{bmatrix}
}_{(A_3)_0},
\quad
\underbrace{
\begin{bmatrix}
0 & 0 \\
0 & \bullet
\end{bmatrix}
}_{(A_3)_1},
\quad
\underbrace{
\begin{bmatrix}
0 & \bullet \\
\bullet & \ast
\end{bmatrix}
}_{(A_3)_2},
\quad
\underbrace{
\begin{bmatrix}
\bullet & \ast \\
\ast & \ast
\end{bmatrix}
}_{(A_3)_3},
\quad
\ldots
\bigrightparenthesis,
\]
\[
A_4=
\bigleftparenthesis
\
\underbrace{
\begin{bmatrix}
0
\end{bmatrix}
}_{(A_4)_{-1}},
\quad
\underbrace{
\begin{bmatrix}
0 & 0 \\
0 & 0
\end{bmatrix}
}_{(A_4)_0},
\quad
\underbrace{
\begin{bmatrix}
0 & 0 \\
0 & 0
\end{bmatrix}
}_{(A_4)_1},
\quad
\underbrace{
\begin{bmatrix}
0 & 0 \\
0 & \bullet
\end{bmatrix}
}_{(A_4)_2},
\quad
\underbrace{
\begin{bmatrix}
0 & \bullet \\
\bullet & \ast
\end{bmatrix}
}_{(A_4)_3},
\quad
\ldots
\bigrightparenthesis.
\]
\begin{itemize}
\item Let $\xi=-1$ and $\eta=0$.
Matrices $G_0 \eqdef(A_1)_0$ and $G_1\eqdef (A_2)_0$ satisfy conditions of Lemma~\ref{lem:construct_basic_matrices_from_generators}.
Given $p$ in $\{0,1\}$,
we apply the recipe from Lemma~\ref{lem:construct_basic_matrices_from_generators} to these matrices and obtain the basic matrices $\E_{p,p}$ in $\Mat_2$:
\[
\E_{0,0}
=\bigl(\nu_{0,0} (A_1)_0 + \nu_{0,1} (A_2)_0\bigr)\,
(A_2)_0^2\,
\bigl(\nu_{0,0} (A_1)_0 + \nu_{0,1} (A_2)_0\bigr),
\]
\[
\E_{1,1}
=\bigl(\nu_{1,1} (A_2)_0\bigr)\,
(A_2)_0^2\,
\bigl(\nu_{1,1} (A_2)_0\bigr).
\]
Define the whole matrix sequenceces $X,Y\in\cX_{2,\al}$ by the same rule:
\[
X
\eqdef \bigl(\nu_{0,0} A_1 + \nu_{0,1} A_2\bigr)\,
A_2^2\,
\bigl(\nu_{0,0} A_1 + \nu_{0,1} A_2\bigr),
\qquad
Y
\eqdef
\bigl(\nu_{1,1} A_2\bigr)\,
A_2^2\,
\bigl(\nu_{1,1} A_2\bigr).
\]
Then $X_0=\E_{0,0}$ and $Y_0=\E_{1,1}$.
On the other hand, $(A_2)_{-1}$ is the zero matrix in $\Mat_1$,
therefore $X_{-1}$ and $Y_{-1}$ coincide with the zero matrix in $\Mat_1$.

\item
Let $\xi=-1$ and $\eta=1$. Matrices $G_0 \eqdef(A_2)_1$ and $G_1\eqdef (A_3)_1$ satisfy conditions of Lemma~\ref{lem:construct_basic_matrices_from_generators}.
According to this lemma, we apply some operations on these matrices to get the basic matrices $\E_{0,0}$ and $\E_{1,1}$.
However, applying the same operations to $(A_2)_{-1}$ and $(A_3)_{-1}$ we get the zero matrix in $\Mat_1$.

\item
Let $\xi=-1$ and $\eta=2$.
Now, we consider $G_0 \eqdef(A_4)_2$ and $G_1 \eqdef(A_3)_2$. In terms of these, Lemma~\ref{lem:construct_basic_matrices_from_generators} gives a recipe to write $\E_{0,0}$ and $\E_{1,1}$. Applying the same recipe to $(A_4)_{-1}$ and $(A_3)_{-1}$ we get in both cases the zero matrix in $\Mat_1$.
\end{itemize}
\end{example}

\smallskip

\begin{example}
\label{example:gammas_negative_3}
Let $n=3$.
Consider the matrix sequences
$A_p\eqdef\ga(g_p)$
for $p\in\{2,3,4,5,6\}$.
By Lemma~\ref{lem:gamma_xi_associated_to_Q},
the first elements of these matrix sequences have the following form:
\[
A_2=
\bigleftparenthesis
\
\underbrace{
\begin{bmatrix}
\bullet
\end{bmatrix}
}_{(A_2)_{-2}},
\quad 
\underbrace{
\begin{bmatrix}
0 & \bullet \\
\bullet & \ast
\end{bmatrix}
}_{(A_2)_{-1}},
\quad
\underbrace{
\begin{bmatrix}
0 & 0 & \bullet \\
0 & \bullet & \ast \\
\bullet & \ast & \ast
\end{bmatrix}
}_{(A_2)_0},
\quad
\underbrace{
\begin{bmatrix}
0 & \bullet & \ast \\
\bullet & \ast & \ast \\
\ast & \ast & \ast
\end{bmatrix}
}_{(A_2)_1},
\quad
\underbrace{
\begin{bmatrix}
\bullet & \ast & \ast \\
\ast & \ast & \ast \\
\ast & \ast & \ast
\end{bmatrix}
}_{(A_2)_2},\quad
\ldots
\bigrightparenthesis,
\]
\[
A_3=
\bigleftparenthesis
\
\underbrace{
\begin{bmatrix}
0
\end{bmatrix}
}_{(A_3)_{-2}},
\quad 
\underbrace{
\begin{bmatrix}
0 & 0 \\
0 & \bullet
\end{bmatrix}
}_{(A_3)_{-1}},
\quad
\underbrace{
\begin{bmatrix}
0 & 0 & 0 \\
0 & 0 & \bullet \\
0 & \bullet & \ast
\end{bmatrix}
}_{(A_3)_0},
\quad
\underbrace{
\begin{bmatrix}
0 & 0 & \bullet \\
0 & \bullet & \ast \\
\bullet & \ast & \ast
\end{bmatrix}
}_{(A_3)_1},
\quad
\underbrace{
\begin{bmatrix}
0 & \bullet & \ast \\
\bullet & \ast & \ast \\
\ast & \ast & \ast
\end{bmatrix}
}_{(A_3)_2},\quad
\ldots
\bigrightparenthesis,
\]

\[
A_4=
\bigleftparenthesis
\
\underbrace{
\begin{bmatrix}
0
\end{bmatrix}
}_{(A_4)_{-2}},
\quad 
\underbrace{
\begin{bmatrix}
0 & 0 \\
0 & 0
\end{bmatrix}
}_{(A_4)_{-1}},
\quad
\underbrace{
\begin{bmatrix}
0 & 0 & 0 \\
0 & 0 & 0 \\
0 & 0 & \bullet
\end{bmatrix}
}_{(A_4)_0},
\quad
\underbrace{
\begin{bmatrix}
0 & 0 & 0 \\
0 & 0 & \bullet \\
0 & \bullet & \ast
\end{bmatrix}
}_{(A_4)_1},
\quad
\underbrace{
\begin{bmatrix}
0 & 0 & \bullet \\
0 & \bullet & \ast \\
\bullet & \ast & \ast
\end{bmatrix}
}_{(A_4)_2},\quad 
\ldots
\bigrightparenthesis,
\]

\[
A_5=
\bigleftparenthesis
\
\underbrace{
\begin{bmatrix}
0
\end{bmatrix}
}_{(A_5)_{-2}},
\quad 
\underbrace{
\begin{bmatrix}
0 & 0 \\
0 & 0
\end{bmatrix}
}_{(A_5)_{-1}},
\quad
\underbrace{
\begin{bmatrix}
0 & 0 & 0 \\
0 & 0 & 0 \\
0 & 0 & 0
\end{bmatrix}
}_{(A_5)_0},
\quad
\underbrace{
\begin{bmatrix}
0 & 0 & 0 \\
0 & 0 & 0 \\
0 & 0 & \bullet
\end{bmatrix}
}_{(A_5)_1},
\quad
\underbrace{
\begin{bmatrix}
0 & 0 & 0 \\
0 & 0 & \bullet \\
0 & \bullet & \ast
\end{bmatrix}
}_{(A_5)_2},\quad
\ldots
\bigrightparenthesis,
\]

\[
A_6=
\bigleftparenthesis
\
\underbrace{
\begin{bmatrix}
0
\end{bmatrix}
}_{(A_6)_{-2}},
\quad 
\underbrace{
\begin{bmatrix}
0 & 0 \\
0 & 0
\end{bmatrix}
}_{(A_6)_{-1}},
\quad
\underbrace{
\begin{bmatrix}
0 & 0 & 0 \\
0 & 0 & 0 \\
0 & 0 & 0
\end{bmatrix}
}_{(A_6)_0},
\quad
\underbrace{
\begin{bmatrix}
0 & 0 & 0 \\
0 & 0 & 0 \\
0 & 0 & 0
\end{bmatrix}
}_{(A_6)_1},
\quad
\underbrace{
\begin{bmatrix}
0 & 0 & 0 \\
0 & 0 & 0 \\
0 & 0 & \bullet
\end{bmatrix}
}_{(A_6)_2},\quad
\ldots
\bigrightparenthesis.
\]

\begin{itemize}
\item For $\xi=-2$, $\eta=-1$,
we put $G_0\eqdef (A_2)_{-1}$,
$G_1\eqdef (A_3)_{-1}$.
Then $G_0,G_1$ satisfy conditions
of Lemma~\ref{lem:construct_basic_matrices_from_generators}.
By the rules from that lemma,
we can apply some algebraic operations to $G_0$ and $G_1$ and obtain matrices $\E_{0,0}$ and $\E_{1,1}$ in $\Mat_2$.
The same operations,
applied to $(A_2)_{-2}$ and $(A_3)_{-2}$,
yield the zero matrix in $\Mat_1$.

\item For $\xi=-1$, $\eta=1$,
the matrices
$(A_3)_1$, $(A_4)_1$, $(A_5)_1$
satisfy conditions of Lemma~\ref{lem:construct_basic_matrices_from_generators}.
By the recipe of Lemma~\ref{lem:construct_basic_matrices_from_generators},
we can obtain every basic matrix in $\Mat_3$
from $(A_3)_1$, $(A_4)_1$, $(A_5)_1$.
On the other hand, applying the same operations to
$(A_2)_{-1}$, $(A_3)_{-1}$, $(A_4)_{-1}$,
we get just the zero matrix in $\Mat_2$.

\medskip
\item For $\xi=-1$, $\eta=2$,
the matrices $(A_4)_2$, $(A_5)_3$, $(A_6)_2$
satisfy conditions of Lemma~\ref{lem:construct_basic_matrices_from_generators}.
By the scheme proposed there, one can obtain from such matrices every basic matrix in $\Mat_3$.
On the other hand, the same operations
applied to $(A_4)_{-1}$, $(A_5)_{-1}$, $(A_6)_{-1}$
yield only the zero matrix in $\Mat_2$.

\medskip
\item For $\xi=-2$, $\eta=1$,
matrices $G_0 \eqdef(A_3)_1$, $G_1\eqdef (A_4)_1$, and $G_3\eqdef (A_5)_1$ satisfy conditions of Lemma~\ref{lem:construct_basic_matrices_from_generators}.
From these, we get all diagonal basic matrices in $\Mat_3$.
Applying the same operations to $(A_3)_{-2}$, $(A_4)_{-2}$, and $(A_5)_{-2}$ we obtain the zero matrix in $\Mat_1$.

\end{itemize}
\end{example}

Examples~\ref{example:gammas_negative_2} and~\ref{example:gammas_negative_3} are useful to understand the forthcoming general proofs.

\begin{prop}[separation of pure states associated to different negative frequencies]
\label{prop:separate_pure_states_associated_to_different_negative_frequencies}
Let $\xi,\eta\in\{-n+1,\ldots,-1\}$
such that $\xi<\eta$.
Let $u\in\bS_{n+\xi}$
and $v\in\bS_{n+\eta}$.
Then there exists $X$ in $\cX_{n,\al}$ such that
\[
\purestate_{\xi,u}(X)
\ne\purestate_{\eta,v}(X).
\]
\end{prop}

\begin{proof}
The assumption $\xi<\eta<0$
implies that $|\eta|<|\xi|$.
For every $k$ in $\bNz$, we put
\[
A_k \eqdef \ga(g_k).
\]
For each $q$ in $\{0,\ldots,n-1+\eta\}$,
consider the matrix
\[
G_q \eqdef (A_{n-1+q})_\eta.
\]
These matrices belong to $\Mat_{n+\eta}$.
According to Lemma~\ref{lem:gamma_xi_associated_to_Q},
$G_q$ is $(n+\eta-1+q)$-antitriangular.
By Remark~\ref{rem:antitriangular_serve_as_generators},
$G_0,\ldots,G_{n+\eta-1}$
satisfy conditions of Lemma~\ref{lem:construct_basic_matrices_from_generators},
with $n+\eta$ instead of $n$.

Let $p\in\{0,\ldots,n+\eta-1\}$ such that $v_p\ne 0$.
Then, by Lemma~\ref{lem:construct_basic_matrices_from_generators} and Remark~\ref{rem:construct_basic_matrices_from_real_symmetric_generators},
there exist $\nu_{p,p},\ldots,\nu_{p,n+\eta-1}$ in $\bR$ such that
\begin{equation}
\label{eq:Epp_via_Aeta_with_eta_negative}
\E_{p,p}
=
\left(\,
\sum_{j=p}^{n+\eta-1}
\nu_{p,j} (A_{n-1+j})_\eta
\right)
(A_{2n-2+\eta})_\eta^2
\left(\,
\sum_{j=p}^{n+\eta-1} \nu_{p,j}
(A_{n-1+j})_\eta
\right).
\end{equation}
We construct $X\in\cX_{n,\al}$ applying the same operations to the whole sequences $A_{n-1+j}$:
\begin{equation}
\label{eq:define_X_via_A_eta_negative}
X \eqdef 
\left(\,\sum_{j=p}^{n+\eta-1}
\nu_{p,j} A_{n-1+j}\right)
A_{2n-2+\eta}^2
\left(\,\sum_{j=p}^{n+\eta-1}
\nu_{p,j} A_{n-1+j}\right).
\end{equation}
Then, by~\eqref{eq:Epp_via_Aeta_with_eta_negative},
$X_\eta = \E_{p,p}$ and
\[
\purestate_{\eta,v}(X)
=\langle X_\eta v,v\rangle
=\langle \E_{p,p} v, v\rangle
=|v_p|^2 > 0.
\]
On the other hand,
by Lemma~\ref{lem:when_gamma_g_is_zero},
$(A_{2n+\eta-2})_\xi$ is the zero matrix in $\Mat_{n+\xi}$ because
\[
2(n+\xi)-2
<
2n+\eta-2-|\xi|.
\]
The formula for $X_\xi$ contains 
$(A_{2n+\eta-2})_\xi$ as a factor.
Therefore, $X_\xi$ is the zero matrix and $\sigma_{\xi,u}(X)=0$.
\end{proof}

\smallskip

\begin{prop}[separation of pure states associated to $\xi$ and $\eta$ with $-\eta\leq\xi<0$]
\label{prop:separate_pure_states_associated_to_different_frequences_xi_eq_minus_eta}
Let $\xi,\eta \in \Om_n$ such that $-\eta\leq\xi<0$.
Let $u\in\bS_{n+\xi}$,
$v\in\bS_{n}$.
Then there exists $X$ in $\cX_{n,\al}$ such that
\[
\purestate_{\xi,u}(X)
\ne\purestate_{\eta,v}(X).
\]
\end{prop}

\begin{proof}
For every $k$ in $\bNz$, set $A_k \eqdef \ga(g_k)$.
By Lemma~\ref{lem:gamma_xi_associated_to_Q}, for each $q$ in $\{0,\dots, n-1\}$,
\[
G_{q}\eqdef (A_{n-1+\eta+q})_\eta
\]
is $(n-1+q)$-antitriangular. Thus,
$G_0,\ldots,G_{n-1}$
satisfy the requirements of Lemma~\ref{lem:construct_basic_matrices_from_generators}.
Let $p\in \{0,\ldots,n-1\}$ such that $v_p\ne 0$. By the same rule as in the proof of Proposition~\ref{prop:separate_pure_states_associated_to_different_positive_frequencies}, we can construct $X\in\cX_{n,\al}$ such that
$X_{\eta}=\E_{p,p}$ and therefore $\sigma_{\eta,v}(X)\ne0$.
Let us show that $\purestate_{\xi,u}(X)=0$.
By Lema~\ref{lem:when_gamma_g_is_zero},
$(A_{2n-2+\eta})_{\xi}$ is the zero matrix in $\Mat_{n+\xi}$ because
\[
2(n+\xi)-2
<2n-2+\eta+\xi
=2n-2+\eta-|\xi|.
\]
Since $A_{2n-2+\eta}$ appears as a factor in the expression~\eqref{eq:define_X_via_A} for $X$,
we conclude that $X_{\xi}$ is the zero matrix in $\Mat_{n+\xi}$ and $\purestate_{\xi,u}(X)=0$.
\end{proof}

\begin{prop}[separation of pure states associated to $\xi$ and $\eta$ with $\eta\ge0$ and $\xi<-\eta$]
\label{prop:separate_pure_states_associated_to_different_frequences_xi_less_minus_eta}
Let $\xi,\eta\in\Om_n$ such that
\[
0\le\eta\le n-2,\qquad
-n+1\le \xi<-\eta.
\]
Let $u\in\bS_{n+\xi}$ and $v\in\bS_n$.
Then there exists $X$ in $\cX_{n,\al}$ such that
\[
\purestate_{\xi,u}(X)
\ne\purestate_{\eta,v}(X).
\]
\end{prop}

\begin{proof}
The assumptions on $\xi$ and $\eta$
imply that $\eta<|\xi|$.
For every $k$ in $\bNz$, we take
$A_k \eqdef \ga(g_k)$ as before.
For each $q$ in $\{0,\ldots,n-1\}$,
consider the matrix
\[
G_q \eqdef (A_{n-1+\eta+q})_\eta\in\Mat_n.
\]
By Lemma~\ref{lem:gamma_xi_associated_to_Q},
$G_q$ is $(n-1+q)$-antitriangular.
Therefore,
$G_0,\ldots,G_{n-1}$
satisfy conditions of Lemma~\ref{lem:construct_basic_matrices_from_generators}.

Let $p\in\{0,\ldots,n-1\}$ such that $v_p\ne 0$.
Similarly to the proof of Proposition~\ref{prop:separate_pure_states_associated_to_different_positive_frequencies},
we construct $X\in\cX_{n,\al}$ by the rule~\eqref{eq:define_X_via_A}. Hence,
$X_\eta=\E_{p,p}$ and therefore $\sigma_{\eta,v}(X)\ne0$.

On the other hand,
by Lemma~\ref{lem:when_gamma_g_is_zero}, $(A_{2n-2+\eta})_\xi$ is the zero matrix in $\Mat_{n+\xi}$ because
\[
2(n+\xi) - 2
< 2n - 2 + \eta + \xi
= 2n - 2 + \eta - |\xi|.
\]
Therefore, $X_\xi$ is the zero matrix and $\sigma_{\xi,u}(X)=0$.
\end{proof}

\section{Proofs of the main results}
\label{sec:proofs_of_main_results}

In this section,
we prove Theorems~\ref{thm:main} and~\ref{thm:closure_is_not_Cstar_algebra}.
We also prove that $\cG_{1,\al}$ is dense in $c(\bNz)$.

\begin{proof}[Proof of Theorem~\ref{thm:main}.]
By Proposition~\ref{prop:Ln_is_liminal},
$\cL_n$ is a unital C*-algebra of type I.
Their pure states are described in Proposition~\ref{prop:pure_states_of_Ln}.
Furthermore, by Proposition~\ref{prop:cX_subseteq_cL},
$\cX_{n,\al}$ is a C*-subalgebra of $\cL_n$.
Propositions~\ref{prop:separation_pure_states_same_frequency},
\ref{prop:purestate_indicator_function},
\ref{prop:separate_pure_states_associated_to_different_positive_frequencies},
\ref{prop:separate_pure_states_associated_to_different_negative_frequencies},
\ref{prop:separate_pure_states_associated_to_different_frequences_xi_eq_minus_eta},
\ref{prop:separate_pure_states_associated_to_different_frequences_xi_less_minus_eta}
show that $\cX_{n,\al}$ separates the pure states of $\cL_n$.
By Theorem~\ref{thm:Kaplansky},
$\cX_{n,\al}=\cL_n$.
\end{proof}

Now we will show that for $n\ge 2$ and some $\xi,\eta,u,v$ with $\xi\ne\eta$,
the generating class
$\cG_{n,\al}$ is not sufficient to separate $\purestate_{\xi,u}$ and $\purestate_{\eta,v}$.

\begin{prop}
\label{prop:pure_states_can_coincide_on_gammas}
Let $n\in\bN$, $n\ge 2$, $\xi=0$, $\eta=2$.
Then there exist $u,v$ in $\bS_n$
such that for every $a$ in $\BL$,
\[
\purestate_{\xi,u}(\ga(a))
=\purestate_{\eta,v}(\ga(a)).
\]
\end{prop}

\begin{proof}
We define $u$ and $v$ by
\[
u
=\left[\,\sqrt{\frac{\al+3}{2(\al+2)}},\ \sqrt{\frac{\al+1}{2(\al+2)}},\,0,\,\ldots,\,0\right]^\top,\qquad
v
=e_0
=\left[1,0,\ldots,0\right]^\top.
\]
Then $u,v\in\bS_n$.
Let $a\in\BL$.
We put $A\eqdef\ga(a)$
and represent $\purestate_{0,u}(A)$ and $\purestate_{2,v}(A)$ as in Lemma~\ref{lem:purestate_via_integral}:
\[
\purestate_{0,u}(A)
=
\int_0^1 a(\sqrt{t})\,|F_{\al,0,u}(t)|^2\,\dif{}t,
\qquad
\purestate_{2,v}(A)
=
\int_0^1 a(\sqrt{t})\,|F_{\al,2,v}(t)|^2\,\dif{}t.
\]
By~\eqref{eq:shifted_Jacobi_explicit},\eqref{eq:jac}, and \eqref{eq:jaccoef_via_binomial},
\[
\jac_0^{(\al,0)}(t)
=
\sqrt{\al + 1}\,(1 - t)^{\al/2},
\qquad
\jac_1^{(\al,0)}(t)
=\sqrt{\al + 3}\,(1 - t)^{\al/2} ((\al + 2)t -1),
\]
\[
\jac_0^{(\al,2)}(t)
=\sqrt{\displaystyle\frac{(\al+3)(\al+2)(\al+1)}{2}}\,(1 - t)^{\al/2}t.
\]
Thus, in our case, for every $t$ in $[0,1)$,
\begin{align*}
F_{\al,0,u}(t)
&=u_0 \jac_0^{(\al,0)}(t)
+u_1 \jac_1^{(\al,0)}(t)
\\[0.5ex]
&=
(1-t)^{\al/2}
\left(
\sqrt{\frac{(\al+1)(\al+3)}{2(\al+2)}}
+\sqrt{\frac{(\al+1)(\al+3)}{2(\al+2)}}
\,\bigl((\al+2)t-1)\bigr)
\right)
\\[0.5ex]
&=
\sqrt{\displaystyle\frac{(\al+3)(\al+2)(\al+1)}{2}}(1 - t)^{\al/2}t,
\end{align*}
and
\[
F_{\al,2,v}(t)
=\jac_0^{(\al,2)}(t)
=\sqrt{\displaystyle\frac{(\al+3)(\al+2)(\al+1)}{2}}(1 - t)^{\al/2}t.
\]
Hence, $F_{\al,0,u}=F_{\al,2,v}$
and
$\purestate_{0,u}(A)=\purestate_{2,u}(A)$.
\end{proof}

\begin{remark}
There are analogs of Proposition~\ref{prop:pure_states_can_coincide_on_gammas} for some other values of $\xi$ and $\eta$.
For example,
if $\eta\in\{1,\ldots,n-1\}$,
$\xi=-\eta$,
and $u=v=e_0$, then
$\purestate_{\xi,u}(\ga(a))
=\purestate_{\eta,v}(\ga(a))$ for every $a$ in $\BL$,
because $\ga(a)_\xi$ is just a leading principal submatrix of $\ga(a)_\eta$.
\end{remark}

\begin{proof}[Proof of Theorem~\ref{thm:closure_is_not_Cstar_algebra}]
Let $\xi=0$, $\eta=2$,
and $u,v\in\bS_n$ be as in Proposition~\ref{prop:pure_states_can_coincide_on_gammas}.
For every $A$ in $\cG_{n,\al}$,
\[
\purestate_{\xi,u}(A)
=\purestate_{\eta,v}(A).
\]
Since $\purestate_{\xi,u}$ and $\purestate_{\eta,v}$ are continuous function on $\cL_n$,
this equality holds for all $A$ in the closure of $\cG_{n,\al}$.
On the other hand, 
$\purestate_{\xi,u}$ and $\purestate_{\eta,v}$ are different elements of $\cL_n$, and they are separated by $\cL_n$.
More explicitly,
let $X_\eta\eqdef I_n$
and $X_j$ be the zero matrix for $j$ in $\Om_n\setminus\{\eta\}$.
Then $X\in\cL_n$,
$\purestate_{\xi,u}(X)=0$,
and $\purestate_{\eta,v}(X)=1$.
Hence, $X\in\cL_n\setminus\operatorname{clos}(\cG_{n,\al})$.
\end{proof}

\medskip
For $n=1$, the situation is simpler.

\begin{prop}
\label{prop:G1_is_dense_in_c}
For $n=1$, $\cG_{1,\al}$ is a dense subset of $c(\bNz)$.
\end{prop}

\begin{proof}
If $\om\in\bC$ and $a(r)=\om$ for every $r$ in $[0,1)$,
then $\ga_{1,\al}(a)_\xi=\om$ for every $\xi$ in $\bNz$.
Therefore, we consider the class of generating symbols
\[
\BLZ\eqdef
\left\{
a\in L^\infty([0,1))\colon\quad
\lim_{r\to 1}a(r)=0
\right\},
\]
and we will prove that
$\cG_{1,\al,0}\eqdef\{\ga_{1,\al}(a)\colon\ a\in\BLZ\}$ is a dense subset of $c_0(\bNz)$.
We have not found a recipe for a constructive approximation.
Instead of that, we use the following argument, similar to
\cite[proof of Theorem~4.8]{EsmeralMaximenkoVasilevski2015}
and
\cite[proof of Theorem~7.3]{EsmeralMaximenko2016}.
If the closure of $\cG_{1,\al,0}$ is not equal to $c_0(\bNz)$, then, by Hahn--Banach theorem, there exists a nonzero bounded linear functional on $c_0(\bNz)$ vanishing on $\cG_{1,\al,0}$.
Every bounded linear functional on $c_0(\bNz)$ 
can be represented as
\[
\ph_b(x)=\sum_{\xi=0}^\infty b_\xi x_\xi\qquad(x\in c_0(\bNz)),
\]
for some $b$ in $\ell^1(\bNz)$.
So, we suppose that $b\in\ell^1(\bNz)$ and for every $a$ in $\BLZ$,
\begin{equation}
\label{eq:sum_b_ga_is_zero}
\sum_{\xi=0}^\infty b_\xi \ga_{1,\al}(a)_\xi=0.
\end{equation}
Our goal is to prove that $b$ is the zero sequence.
We consider an auxiliar function $f\colon\bD\to\bC$,
\[
f(z)
\eqdef 
\sum_{\xi=0}^\infty
b_\xi\,
\frac{\Ga(\al+\xi+2)}{\Ga(\al+1)\,\xi!}\,
z^\xi.
\]
Stirling's formula and condition $b\in\ell^1(\bNz)$ imply that the series absolutely converges in the unit disk $\bD$.
Therefore,
$f$ is an analytic function on $\bD$.

For every $s$ in $(0,1)$,
we apply \eqref{eq:sum_b_ga_is_zero}
to $a=1_{[0,s)}$:
\[
0=
\sum_{\xi=0}^\infty
b_\xi\,
\frac{\Ga(\al+\xi+2)}{\Ga(\al+1)\,\xi!}
\int_0^{s^2}
(1-t)^\al t^\xi\,\dif{}t
=
\sum_{\xi=0}^\infty
\int_0^{s^2}
b_\xi\,
\frac{\Ga(\al+\xi+2)}{\Ga(\al+1)\,\xi!}
(1-t)^\al t^\xi\,\dif{}t
.
\]
The last series of integrals satisfies
\[
\sum_{\xi=0}^\infty
\int_0^{s^2}
\left|
b_\xi\,
\frac{\Ga(\al+\xi+2)}{\Ga(\al+1)\,\xi!}
(1-t)^\al t^\xi
\right|\dif{}t
\le
\sum_{\xi=0}^\infty
|b_\xi|
=\|b\|_1
<+\infty.
\]
Therefore, by a corollary from Lebesgue's dominated convergence theorem
(or by Fubini's theorem),
the integral interchanges with the series:
\[
0=
\int_0^{s^2}
\left(\,
\sum_{\xi=0}^\infty
b_\xi\,
\frac{\Ga(\al+\xi+2)}{\Ga(\al+1)\,\xi!}
(1-t)^\al t^\xi\right)
\dif{}t
=
\int_0^{s^2}
(1-t)^\al f(t)\,\dif{}t.
\]
By the first fundamental theorem of calculus,
$f(t)=0$ for every $t$ in $(0,1)$.
Therefore, all coefficients $b_\xi$ are zero.
\end{proof}

\section*{Acknowledgements}

The authors are grateful to Professor Armando S\'{a}nchez-Nungaray for explaining us many ideas in Section~\ref{sec:algebra_of_matrix_sequences_with_scalar_limits}.

\bigskip\noindent
Roberto Mois\'{e}s Barrera-Castel\'{a}n\newline
Instituto Polit\'{e}cnico Nacional\newline
Escuela Superior de F\'{i}sica y Matem\'{a}ticas\newline
Ciudad de M\'{e}xico\newline
Mexico\newline
e-mail: rmoisesbarrera@gmail.com \newline
https://orcid.org/0000-0001-9549-3482

\bigskip\noindent
Egor A. Maximenko\newline
Instituto Polit\'{e}cnico Nacional\newline
Escuela Superior de F\'{i}sica y Matem\'{a}ticas\newline
Ciudad de M\'{e}xico\newline
Mexico\newline
e-mail: egormaximenko@gmail.com, emaximenko@ipn.mx\newline
https://orcid.org/0000-0002-1497-4338

\bigskip\noindent
Gerardo Ramos-Vazquez\newline
Universidad Veracruzana\newline
Facultad de Matem\'{a}ticas\newline
Xalapa, Veracruz\newline
Mexico\newline
e-mail: ger.ramosv@gmail.com\newline
https://orcid.org/0000-0001-9363-8043

\end{document}